\documentclass[a4paper,10pt]{amsart}

  \usepackage{amsmath}		   			                
  \usepackage{amsthm}					                
  \usepackage{amssymb}  				    	        
  \usepackage{mathrsfs}					                
  \usepackage{color}					    	        
  \usepackage[all]{xy}  	  		                    
  \usepackage{fancyhdr}                             	
  \usepackage[shortlabels]{enumitem}					
  
  \setenumerate[1]{itemsep=0pt,topsep=0pt}
  
  \usepackage{hyperref}
  
  \theoremstyle{remark}         \newtheorem{remark}{Remark}[section]
  \theoremstyle{remark} 		
  \theoremstyle{definition}		\newtheorem{definition}{Definition}[section]
  \theoremstyle{plain}			\newtheorem{theorem}{Theorem}[section]
  \theoremstyle{plain}			\newtheorem{lemma}[theorem]{Lemma}
  \theoremstyle{plain}			\newtheorem{corollary}[theorem]{Corollary}
  \theoremstyle{plain}			\newtheorem{proposition}[theorem]{Proposition}
  
  \DeclareMathOperator{\im}{im}				        	    
  \DeclareMathOperator{\Hom}{Hom}				            
  \DeclareMathOperator{\Ext}{Ext}				            
  \DeclareMathOperator{\Tor}{Tor}				            
  \DeclareMathOperator{\Aut}{Aut}				            
  \DeclareMathOperator{\id}{id}				                

  \newcommand{\varphantom}[1]{\mathrel{\phantom{#1}}}
  \newcommand{\pl}{\partial}
  
  \newcommand{\To}{\longrightarrow}
  \newcommand{\ot}{\otimes}
  \newcommand{\diffd}{\mathsf{d}}			             	
  
  \newcommand{\al}{\alpha}
  \newcommand{\be}{\beta}
  \newcommand{\ga}{\gamma}
  
  \newcommand{\si}{\sigma}

  \newcommand{\e}{\varepsilon}
  
  \newcommand{\De}{\Delta}
  \newcommand{\lam}{\lambda}
  
  \newcommand{\vphi}{\varphi}

  \newcommand{\mbb}[1]{\mathbb{#1}}			    	
  \newcommand{\mbf}[1]{\mathbf{#1}}			    	
  \newcommand{\mc}[1]{\mathcal{#1}}			    	
  \newcommand{\mf}[1]{\mathfrak{#1}}				
  \newcommand{\kk}{\Bbbk} 			    	        
  \newcommand{\nan}{\mathbb{N}}			    	    
  \newcommand{\inn}{\mathbb{Z}} 				    
  \newcommand{\con}{\mathbb{C}}			    	    

\usepackage{cases}

\allowdisplaybreaks
\numberwithin{equation}{section}

\newcommand{\Dell}{{}^\si\!\De}
\newcommand{\Delr}{\De^\si}
\newcommand{\Delb}{{}^\si\!\De^\si}

\newcommand{\Tot}{\mathrm{Tot}\,}
\newcommand{\nc}{\mathrm{nc}}
\newcommand{\difl}{{}^\si\!\diffd}
\newcommand{\difr}{\diffd^\si\!}
\newcommand{\difb}{{}^\si\!\diffd^\si\!}
\newcommand{\sbu}{{\scriptscriptstyle\bullet}}

\title[On homological smoothness of generalized Weyl algebras]{On homological smoothness of generalized Weyl algebras over polynomial algebras in two variables}

\author[Liyu Liu]{Liyu Liu}
\address{School of Mathematical Science, Yangzhou University, No.\ 180 Siwangting Road, 225002 Yangzhou, Jiangsu, China}
\email{lyliu@yzu.edu.cn}

\thanks{The author acknowledges the supports of the Natural Science Foundation of China No.\ 11501492, the Natural Science Foundation of Jiangsu Province No.\ BK20150435 and the Natural Science Foundation for Universities in Jiangsu Province No.\ 15KJB110022. He is grateful to the referee for the valuable comments. He would like to send his thanks to Xuefeng Mao, Quanshui Wu, and James Zhang for helpful conversations, special thanks to Shengyun Jiao for her Master research that is related to \S\ref{subsec:Mpq}.}

\begin{document}

	\begin{abstract}
		Homological smoothness and twisted Calabi-Yau property of generalized Weyl algebras over polynomial algebras in two variables is studied. A necessary and sufficient condition to be homologically smooth is given. The Nakayama automorphisms of such algebras are also computed in terms of the Jacobian determinants of defining automorphisms. 
	\end{abstract}
	\keywords{Generalized Weyl algebra, Homological smoothness, Twisted Calabi-Yau algebra}
\subjclass[2010]{Primary 16S38, 16E10, 16E40}
 	\maketitle
	
\section{Introduction}

Motivated by the study of algebras analogous with the classical Weyl algebra $A_1(\kk)$ over a field $\kk$, Bavula introduced in \cite{Bavula:GWA-def} the notion of (degree one) generalized Weyl algebras over the polynomial algebra $\kk[z]$. Later on, generalized Weyl algebras over any algebra $B$ were defined in \cite{Bavula:GWA-tensor-product}. Roughly speaking, a generalized Weyl algebra over $B$ is an extension of $B$ by two formal variables $x$, $y$, parameterized by an automorphism $\si$ on $B$ and a central element $\vphi\in B$, denoted by $B(\si,\vphi)$. Many people have intensively studied the situation $B=\kk[z]$, and  such generalized Weyl algebras are denoted by $W_{(1)}$ in this paper. It is illustrated in \cite{Bavula:GWA-def} that the global dimensions of $W_{(1)}$ are equal to $1$, $2$, or $\infty$. and the latter occurs if and only if the defining polynomial $\vphi$ admits a multiple root (see also \cite{Bavula:GWA-tensor-product}, \cite{Hodge:Kleinian-singularities}, \cite{Solotar:Hochschild-homology-GWA-quantum}). Their Hochschild homology and cohomology were computed in \cite{Farinati:Hochschid-homology-GWA}, \cite{Solotar:Hochschild-homology-GWA-quantum}. In particular, a remarkable result in \cite{Farinati:Hochschid-homology-GWA} is that to assure a duality between its Hochschild homology and cohomology (with coefficients in $W_{(1)}$ itself), $W_{(1)}$ should have finite projective dimension as a bimodule over itself, or equivalently, be homologically smooth. To answer \cite[Question 2]{Krahmer:qh-space} (see also \cite{L-Shen-Wu:homo-quan-sphere}), the author explained in \cite{L:gwa-def} when $W_{(1)}$ is homologically smooth; a necessary and sufficient condition was given, under which a duality between its Hochschild homology and cohomology with coefficients in any bimodule was established. Global dimensions of generalized Weyl algebras $W_{(n)}$ over $B=\kk[z_1,\ldots,z_n]$ were also examined by Bavula \cite{Bavula:gldim-polyn}.

Among the results mentioned above and others, $B$ is always required to be commutative. This is not incidental. From an algebraic point of view, noncommutative rings are more rigid than commutative ones, namely, a commutative ring possesses ``more'' automorphisms and central elements than a noncommutative ring. A typical example is: if $B$ is the generic quantum 2-plane $\kk_q[z_1,z_2]$, then $\si$ is necessarily given by $\si(z_1)=az_1$, $\si(z_2)=bz_2$ for some nonzero scalars $a$, $b\in\kk$, and the central element $\vphi$ must be in $\kk$. One concludes that in this case any generalized Weyl algebra over $\kk_q[z_1,z_2]$ is isomorphic to a localization of a quantum 3-space. Hence the noncommutativity of $B$ usually makes the research of generalized Weyl algebras trivial. So generalized Weyl algebras over commutative rings are more interesting. According to papers on this topic, especially the papers by Bavula, two spaces are extremely important in researching generalized Weyl algebras. One is the space $\mathrm{MaxSpec}(B)/\langle \si\rangle$ of orbits where $\langle \si\rangle$ is the cyclic group generated by $\si$ acting on the space $\mathrm{MaxSpec}(B)$ of maximum spectrum naturally; the other is the $\si$-stable space $\{\si^n(\vphi)\mid n\in\inn\}$. Normally, many properties, such as global dimensions and irreducible representations, depend on both spaces simultaneously. But there is an exception. It is illustrated in \cite{L:gwa-def} that whether a generalized Weyl algebra $W_{(1)}$ is homologically smooth depends only on $\vphi$, also, its Nakayama automorphism depends only on $\si$.

This paper is a sequel to \cite{L:gwa-def}, in which we plan to study the homological smoothness of generalized Weyl algebras $W_{(2)}$ over the polynomial algebra  $\kk[z_1,z_2]$. Homological smoothness, which is a noncommutative generalization of smoothness for commutative algebras, plays an important role in homological algebra, mathematical physics, etc. An algebra $A$ is called homologically smooth if $A$ admits a bound resolution by finitely generated projective $A^e$-modules (where $A^e:=A\otimes_{\kk}A^{\mathrm{op}}$ is the enveloping algebra), or equivalently, $A$ is isomorphic to a perfect complex in the derived category $\mc{D}^b(A^e\textrm{-}\mathsf{Mod})$. Recall that in \cite{L:gwa-def} our strategy was to construct a free $W_{(1)}^e$-module resolution of $W_{(1)}$ and then to compute cohomology by it. We will follow the idea for $W_{(2)}$ in this paper. Since the notion of homotopy double complex introduced in \cite{L:gwa-def} is useful to construct a free resolution of $W_{(2)}$, we review it briefly in this paper. After that, noncommutative differential 1-forms and derivations on $\kk[z_1,z_2]$ are introduced, and a noncommutative version of Jacobian determinant is defined accordingly. All of these appear in \S\ref{sec:preliminaries}.

After the preliminaries, we begin \S\ref{sec:homo-smooth} by constructing a free resolution of $W_{(2)}$. 
The resolution is given by Proposition \ref{prop:free-resol-GWA}. Using the resolution, we compute the Hochschild cohomology $H^*(W_{(2)}, M)$ with coefficients in any bimodule $M$.  Notice that $\vphi$ under consideration is a polynomial $\vphi(z_1,z_2)$. Let $\vphi_1$, $\vphi_2$ be the two formal partial derivatives of $\vphi$ with respect to $z_1$, $z_2$, and $(\vphi,\vphi_1,\vphi_2)$ be the ideal of $\kk[z_1,z_2]$ generated by the three elements. Our result is that $H^4(W_{(2)}, M)=0$ if $(\vphi,\vphi_1,\vphi_2)=\kk[z_1,z_2]$. On the other hand, inspired by works of Bavula, we prove that $W_{(2)}$ has infinite global dimension if $(\vphi,\vphi_1,\vphi_2)$ is a proper ideal. Therefore, we obtain the main theorem (Theorem \ref{thm:main-theorem}):

\begin{theorem}
	For any $\si\in\Aut(\kk[z_1,z_2])$, $W_{(2)}$ is homologically smooth if and only if $(\vphi,\vphi_1,\vphi_2)=\kk[z_1,z_2]$.
\end{theorem}

Observe that $W_{(n)}$ is commutative if $\si=\id$, and in this case we write it as $\overline{W}_{(n)}$. We remark that if $\kk$ is of characteristic zero, then the condition $(\vphi,\vphi_1,\vphi_2)=\kk[z_1,z_2]$ is exactly equivalent to that  $\overline{W}_{(2)}$ is smooth in the commutative sense. Coincidentally, it is the same case with $\overline{W}_{(1)}$, which has been explained in \cite{L:gwa-def} using deformation theory. We regard $W_{(2)}$ as a deformation of $\overline{W}_{(2)}$ by the automorphism $\si$. The smoothness of $\overline{W}_{(2)}$ is equivalent to the homological smoothness of $W_{(2)}$. See Remark \ref{rmk:two-smoothness}.

The Nakayama automorphism of an algebra is an important invariant. It is related with the study of rigid dualizing complexes, twisted Poincar\'e duality, Hopf algebra actions on Artin-Schelter regular algebras, and so forth. Examples of algebras that have Nakayama automorphisms are: noetherian Artin-Schelter Gorenstein algebras, many noetherian Hopf algebras, some (co)invariant subalgebras of Artin-Schelter regular algebras, PBW deformations of some graded algebras. In general, the computation of Nakayama automorphisms is not easy. We refer to \cite{Bocklandt-Schedler-Wemyss:superpotential}, \cite{Brown-Zhang:AS-Gorenstein-Hopf}, \cite{Kirkman-Kuzmannovich-Zhang:Nak-rigidity}, \cite{Shen-Lu:Nak-PBW}, \cite{Lv-Mao-Zhang:NAK}, \cite{Lv-Mao-Zhang:Nak-graded}, \cite{Reyes-Rogalski-Zhang:skew-CY-homo-identity}, \cite{Yekutieli:rigid-dualizing-complex-universal-enveloping-algebra} and the references therein for the progress on this topic during the past years. The goal in \S\ref{sec:Nak-auto} is to compute the Nakayama automorphism of $W_{(2)}$. By virtue of the free resolution constructed in \S\ref{sec:homo-smooth}, we compute the group $\Ext^3_{W_{(2)}^e}(W_{(2)}, W_{(2)}^e)$. An explicit formula for Nakayama automorphism is hence obtained. As a consequence, we have (Theorem \ref{thm:3-tCY}, Corollaries \ref{cor:3-CY}, \ref{cor:BV})

\begin{theorem}
	Let $W_{(2)}=\kk[z_1,z_2](\si,\vphi)$ be a generalized Weyl algebra, and $J$ be the Jacobian determinant of $\si$. Then
	\[
	\Ext_{W_{(2)}^e}^i(W_{(2)}, W_{(2)}^e)\cong\begin{cases}
	0, & i\neq 3, \\ W_{(2)}^\nu, & i=3.
	\end{cases}
	\]
	where $\nu\in \Aut(W_{(2)})$ is the Nakayama automorphism, given by
	\[
	\nu(x)=Jx,\quad \nu(y)=J^{-1}y,\quad \nu(z_1)=z_1,\quad \nu(z_2)=z_2.
	\]
	In particular, if $(\vphi,\vphi_1,\vphi_2)=\kk[z_1,z_2]$, then
	\begin{enumerate}
		\item $W_{(2)}$ is twisted $3$-Calabi-Yau;
		\item $W_{(2)}$ is  Calabi-Yau if and only if $J=1$;
		\item the Hochoschild cohomology $HH^{\sbu}(W_{(2)})$ has a Batalin-Vilkovisky structure.
	\end{enumerate}
\end{theorem}

In the final section \S\ref{sec:app}, we study the homological smoothness and Nakayama automorphisms of some concrete algebras, by using the two previous theorems. These algebras contain the quantum groups $\mc{O}_q(SL_2)$, $U(\mf{sl}_2)$, noetherian down-up algebras, quantum lens space, and others. In each case our main theorems give alternative proofs of results from the literature.

\section{Preliminaries}\label{sec:preliminaries}
Throughout, $\kk$ is a field, and all algebras are over $\kk$ unless stated otherwise. Unadorned $\ot$ means $\ot_{\kk}$. Let $A$ be an algebra and $M$ an $A$-bimodule. The group of algebra automorphisms of $A$ is denoted by $\Aut(A)$. For any $\si\in \Aut(A)$, denote by ${}^\si M$ (resp.\ $M^\si$) the left $A$-module (resp.\ right $A$-module) whose ground $\kk$-module is the same with $M$ and whose left (resp.\ right) $A$-action is twisted by $\si$, that is, $a\triangleright  m= \si(a)m$ (resp.\ $m\triangleleft a=m\si(a)$) for any $a\in A$, $m\in M$. 

Let $A^\mathrm{op}$ be the opposite algebra of $A$, and $A^e = A\ot A^\mathrm{op}$ the enveloping algebra of $A$. An $A$-bimodule can be viewed as a left $A^e$-module in a natural way. Recall that $A$ is homologically smooth if $A$ as a left (or equivalently, right) $A^e$-module, admits a finitely generated projective resolution of finite length.

\subsection{Generalized Weyl algebras}

We recall the definition of generalized Weyl algebras given by Bavula in \cite{Bavula:GWA-def}, \cite{Bavula:GWA-tensor-product}.

\begin{definition}
	Suppose $B$ is an algebra. For a central element $\vphi\in B$ and an algebra automorphism $\si\in \Aut(B)$, the associated (degree one) generalized Weyl algebra (GWA for short) $W$ is by definition generated by two variables $x$ and $y$ over $B$ subject to
	\begin{gather*}
	xb = \si(b)x,\quad yb = \si^{-1}(b)y,\quad  \forall b\in B, \\
	yx = \vphi, \quad xy =\si(\vphi).
	\end{gather*}
	The algebra is written as $W = B(\si, \vphi)$. 
\end{definition}

We adopt the convention about the super/sub-scripts of cochain/chain complexes as follows:
\begin{gather*}
\cdots\To C^{i-1}\xrightarrow[\quad]{d^{i-1}} C^i\xrightarrow[\;\;\quad]{d^i} C^{i+1}\To\cdots\\
\cdots\To C_{i+1}\xrightarrow[\;\;\quad]{d_i} C_i\xrightarrow[\quad]{d_{i-1}} C_{i-1}\To\cdots
\end{gather*}
and call a chain complex $C_\sbu$ in an abelian category $\mc A$ \textit{alternate in $[n,+\infty)$} if $d_i=d_{i+2}$ as morphisms in $\mc A$ for all $i\geq n$. So necessarily, $C_i=C_{i+2}=C_{i+4}=\cdots$ for all $i\geq n$. A similar definition is given for cochain complexes.

\begin{proposition}\label{propo:reg-GWA}\cite{L:gwa-def}
	Let $W=B(\sigma,\varphi)$ be a GWA as above. Suppose further that $\vphi$ is a regular element. Then $W$ as a left $W^e$-module can be represented by an alternate complex $C_{\sbu}\in\mathrm{Ch}_{\geq 0}(W^e\textrm{-}\mathsf{Mod})$ in $[1,+\infty)$. Concretely, $C_0=W\ot_BW$, $C_1=(W^{\si}\ot_BW)\oplus(W\ot_B{}^{\si}W)$, $C_2=(W\ot_BW)\oplus(W^{\si}\ot_B{}^{\si}W)$, and
	\begin{alignat*}{2}
	d^C_0(1\ot 1, 0) &= 1\ot x-x\ot 1, &\quad d^C_0(0, 1 \ot 1) &= 1\ot y-y\ot 1,\\
	d^{C}_1(1\ot 1, 0)&=(y \ot 1, 1 \ot x), &\quad d^{C}_1(0, 1 \ot 1) &= (1\ot  y, x\ot 1),\\
	d^{C}_2(1\ot  1, 0) &= (-x\ot 1,1\ot  x), &\quad d^{C}_2(0, 1 \ot 1) &= (1 \ot y, -y\ot 1).
	\end{alignat*}
	The augmentation $C_0\to W$ is given by the multiplication map.
\end{proposition}

\begin{remark}
	Although the second summand of $C_2$ can be simplified to $W\ot_BW$ as the author wrote in \cite{L:gwa-def}, we insist on the expression with double $\si$ so as to make the computation easier in the following part.
\end{remark}

\subsection{Noncommutative differential forms and partial derivations}

Since in this paper we mainly focus on GWAs over the polynomial algebra in two variables, a noncommutative version of differential forms and partial derivations will be introduced first of all. From now on, let $B=\kk[z_1,z_2]$. For any polynomial $g=g(z_1,z_2)\in B$, the noncommutative differential $1$-form $\diffd g$ is defined as $g\ot 1-1\ot g$. The noncommutative partial derivations with respect to $z_1$, $z_2$ are defined as $\kk$-linear maps
\begin{align*}
\De_1\colon B&\To B\ot B & \De_2\colon B&\To B\ot B \\
z_1^{i_1}z_2^{i_2}&\longmapsto \sum_{j=1}^{i_1}z_1^{i_1-j}\ot z_1^{j-1}z_2^{i_2}, & z_1^{i_1}z_2^{i_2}&\longmapsto \sum_{j=1}^{i_2}z_1^{i_1}z_2^{i_2-j}\ot z_2^{j-1}.
\end{align*}

We have
\begin{equation}\label{eq:total-derivation}
\diffd g=\De_1(g)\diffd z_1+\De_2(g)\diffd z_2
\end{equation}
which is a noncommutative analogy of the total derivative formula in calculus. Let $\mu$ be the multiplication of $B$. Then it is easy to check that $\mu\De_1=\pl/\pl z_1$ and $\mu\De_2=\pl/\pl z_2$.

Suppose that $\si\colon B\to B$ is an endomorphism which is determined by $\si(z_1)=f_1(z_1,z_2)$, $\si(z_2)=f_2(z_1,z_2)$. We call the determinant
\[
J_{\nc}=\begin{vmatrix}
\De_1(f_1) & \De_1(f_2) \\[1ex]
\De_2(f_1) & \De_2(f_2) 
\end{vmatrix}
\]
the noncommutative Jacobian determinant of $\si$. If we take the image of each entry by $\mu$, it becomes the usual Jacobian determinant of $\si$,
\[
J=\begin{vmatrix}
\dfrac{\pl f_1}{\pl z_1} & \dfrac{\pl f_2}{\pl z_1} \\[2ex]
\dfrac{\pl f_1}{\pl z_2} & \dfrac{\pl f_2}{\pl z_2}
\end{vmatrix}.
\]

Let $u$, $v$ be any $\kk$-linear maps of $B$, we denote ${}^u\De_i^v=(u\ot v)\circ\De_i$ and ${}^u\diffd^v g=(u\ot v)(\diffd g)=u(g)\ot 1-1\ot v(g)$. By convention, $u$ or $v$ is usually omitted if it is the identity map.

\begin{lemma}\label{lem:NC-diff}
	Let $\si$ be an endomorphism of $B$, and $f_1$, $f_2$ as above. 
	We have
	\[
	\difb z_i=\De_1(f_i)\diffd z_1+\De_2(f_i)\diffd z_2
	\]
	for $i=1$, $2$.
\end{lemma}

\begin{proof}
	Directly from \eqref{eq:total-derivation}.
\end{proof}

\begin{remark}
The elements $\diffd z_i$, $\De_i(g)$, etc.\ can be regarded as in $W^e$ via the embedding $B\ot B\hookrightarrow W\ot W^\mathrm{op}$. The reader will not confuse them.
\end{remark}

\subsection{Homotopy double complexes}

In \cite{L:gwa-def} the author introduced the notion of homotopy double complexes in order to present GWAs $W_{(1)}$. Let us recall the definition.

\begin{definition}
	Suppose that $\mc A$ is an abelian category. Let $\{C^{pq}\}_{p,q\in\inn}$ be a family of objects in $\mc A$ together with morphisms $d_v$, $d_h$, $t$ of degrees $(0, 1)$, $(1, 0)$, $(2,-1)$ respectively. The 4-tuple $(C^{\sbu\sbu},d_v, d_h, t)$ is called a homotopy double cochain complex if the following conditions are fulfilled:
	\begin{align}
	d^2_v &= 0, \label{eq:htpy-def-1}\\
	d_hd_v + d_vd_h &= 0, \label{eq:htpy-def-2}\\
	d^2_h + d_vt + td_v &= 0, \label{eq:htpy-def-3}\\
	d_ht + td_h &= 0, \label{eq:htpy-def-4}\\
	t^2 &= 0. \label{eq:htpy-def-5}
	\end{align}
	A homotopy double chain complex is defined in a similar way.
\end{definition}

The associated total complex $(\Tot C^{\sbu\sbu}, d)$ is defined by
\[
(\Tot C^{\sbu\sbu})^n =\bigoplus_{p+q=n}C^{pq}\text{ and }d = d_v + d_h + t,
\]
which is a generalization of the usual total complex of a usual double complex.

\section{Homological smoothness of a GWA over $\kk[z_1,z_2]$}\label{sec:homo-smooth}

Throughout this section, $B=\kk[z_1,z_2]$, $W_{(2)}=B(\si,\vphi)$ is a GWA over $B$. We will construction a homotopy double complex for $W_{(2)}$ whose total complex is a free $W_{(2)}^e$-resolution of $W_{(2)}$. Using this resolution, a necessary and sufficient condition under which $W_{(2)}$ is homologically smooth is discussed. Recall that $\vphi=\vphi(z_1,z_2)$. Write $\vphi_i=\pl\vphi/\pl z_i$, $i=1$, $2$, and let $(\vphi,\vphi_1,\vphi_2)$ be the ideal in $B$ generated by $\vphi$, $\vphi_1$, $\vphi_2$. In this section we will prove

\begin{theorem}\label{thm:main-theorem}
For any $\si\in\Aut(B)$, $W_{(2)}$ is homologically smooth if and only if $(\vphi,\vphi_1,\vphi_2)=B$.
\end{theorem}

\begin{remark}\label{rmk:two-smoothness}
	Notice that $\overline{W}_{(2)}:=B(\id,\vphi)$ is commutative. When $\mathsf{char}\,\kk=0$, $\overline{W}_{(2)}$ is smooth (in the commutative sense) if and only if $(\vphi,\vphi_1,\vphi_2)=B$. To put it another way: $W_{(2)}$ is homologically smooth if and only if $\overline{W}_{(2)}$ is smooth.
	\begin{enumerate}
		\item A similar phenomenon exists for $W_{(1)}$: $W_{(1)}$ is homologically smooth if and only if $\overline{W}_{(1)}:=\kk[z](\si,\vphi)$ is smooth. This has been explained in deformation theory \cite{L:gwa-def}. Since an automorphism $\si\colon\kk[z]\to\kk[z]$ is necessarily given by $\si(z)=\lam z+\eta$ for some $\lam\in\kk^\times$ and $\eta\in\kk$, one can regard $W_{(1)}$ as a deformation of $\overline{W}_{(1)}$, following Van den Bergh \cite{Van-den-Bergh:Koszul-bimodule-complex}.
		\item Back to $B=\kk[z_1,z_2]$, one does not know the expressions of $\si(z_1)$, $\si(z_2)$ for an arbitrary $\si\in\Aut(B)$, although the van der Kulk theorem reveals the structure of the group $\Aut(B)$, i.e., it decomposes into a coproduct of two subgroups \cite{Dicks:aut-polyn-two}, \cite{McKay-Wang:aut-polyn-two}, \cite{vdKulk:aut-polyn-two}. Thus one cannot say that $W_{(2)}$ is a deformation of $\overline{W}_{(2)}$, unless $\si$ is affine. However, the phenomenon can be summarized in this sentence: $\si$ preserves the (non)smoothness of $W_{(n)}$ for $n=1$, $2$.
	\end{enumerate}
\end{remark}

\begin{remark}
	Differential smoothness is another noncommutative generalization of smoothness. Brzezi\'nski discussed noncommutative calculi for a class of differentially smooth GWA $W_{(1)}$ and $W_{(2)}$ whose defining automorphisms $\si$ are affine \cite{Brzezinski:diff-smooth}. Two kinds of smoothness are compared, and examples of algebras that are differentially but not necessarily homologically smooth are given. A relationship between the two forms of smoothness has not yet been understood.
\end{remark}

\subsection{Construction of homotopy double complex}\label{subsec:constr-homo-double}

Since $B$ admits the following Koszul complex
\[
0\To B\ot B\xrightarrow[\quad]{(\diffd z_2\;\;-\diffd z_1)}( B\ot B)^2\xrightarrow[\quad]{\left(\begin{smallmatrix}
	\diffd z_1\\ \diffd z_2
	\end{smallmatrix}\right)} B\ot B
\]
as a $B$-bimodule resolution via $B\ot B\xrightarrow[\quad]{\mu}B$, we obtain left $W_{(2)}^e$-free resolutions of $W_{(2)}\ot_BW_{(2)}$, $W_{(2)}^{\si}\ot_BW_{(2)}$, $W_{(2)}\ot_B{}^{\si}W_{(2)}$, and $W_{(2)}^{\si}\ot_B{}^{\si}W_{(2)}$ as follows:
\begin{align*}
0\To W_{(2)}^e\xrightarrow[\quad]{(\diffd z_2\;\;-\diffd z_1)}(W_{(2)}^e)^2\xrightarrow[\quad]{\left(\begin{smallmatrix}
	\diffd z_1\\ \diffd z_2
	\end{smallmatrix}\right)} W_{(2)}^e &\To W_{(2)}\ot_BW_{(2)}\To 0, \\
0\To W_{(2)}^e\xrightarrow[\quad]{(\difl z_2\;\;-\difl z_1)}(W_{(2)}^e)^2\xrightarrow[\quad]{\left(\begin{smallmatrix}
	\difl z_1\\ \difl z_2
	\end{smallmatrix}\right)} W_{(2)}^e &\To W_{(2)}^{\si}\ot_BW_{(2)}\To 0, \\
0\To W_{(2)}^e\xrightarrow[\quad]{(\difr z_2\;\;-\difr z_1)}(W_{(2)}^e)^2\xrightarrow[\quad]{\left(\begin{smallmatrix}
	\difr z_1\\ \difr z_2
	\end{smallmatrix}\right)} W_{(2)}^e & \To W_{(2)}\ot_B{}^{\si}W_{(2)}\To 0, \\
0\To W_{(2)}^e\xrightarrow[\quad]{(\difb z_2\;\;-\difb z_1)}(W_{(2)}^e)^2\xrightarrow[\quad]{\left(\begin{smallmatrix}
	\difb z_1\\ \difb z_2
	\end{smallmatrix}\right)} W_{(2)}^e & \To W_{(2)}^{\si}\ot_B{}^{\si}W_{(2)}\To 0.
\end{align*}

We will construct a double complex $(\mc{P}_{\sbu\sbu}, d^h,d^v,t)$ in the next step. To be more intuitive, we draw a diagram to illustrate our construction
\[
\xymatrix{
	& \mc P_{02} \ar[d] & \mc P_{12} \ar[d]\ar[l] & \mc P_{22} \ar[d]\ar[l] & \mc P_{32} \ar[d]\ar[l] & \cdots \ar[l] \\
	& \mc P_{01} \ar[d] & \mc P_{11} \ar[d]\ar[l] & \mc P_{21} \ar[d]\ar[l]\ar[llu]|\hole & \mc P_{31} \ar[d]\ar[l]\ar[llu]|\hole & \cdots \ar[l] \\
	& \mc P_{00} \ar[d]\ar@{-->}[ld]_-{\mu} & \mc P_{10} \ar[d]\ar[l] & \mc P_{20} \ar[d]\ar[l]\ar[llu]|\hole & \mc P_{30} \ar[d]\ar[l]\ar[llu]|\hole & \cdots \ar[l] \\
	W_{(2)} & C_0 \ar[l] & C_1 \ar[l] & C_2 \ar[l] & C_3 \ar[l] & \cdots \ar[l]
}
\]
where $C_\sbu$ is the complex given in Proposition \ref{propo:reg-GWA},   $\mc{P}_{i\sbu}$ is a projective resolution of $C_i$ for each $i$, and the dashed arrow $\mc P_{00}\to W_{(2)}$ is the composition $\mc P_{00}\to C_0\to W_{(2)}$, equal to the multiplication map $\mu$.

Based on the alternate complex $C_\sbu$ in Proposition \ref{propo:reg-GWA}, we erect the four resolutions, and then obtain the embryo of a homotopy double complex:
\begin{align*}
\mc P_{00}&=\mc P_{02}=W_{(2)}^e,\\
\mc P_{01}&=\mc P_{10}=\mc P_{12}=\mc P_{20}=\mc P_{22}=\cdots=(W_{(2)}^e)^2,\\
\mc P_{11}&=\mc P_{21}=\mc P_{31}=\cdots=(W_{(2)}^e)^4,
\end{align*}
and all other entries are zero. Moreover, the morphisms $d^v$ are expressed by
\begin{alignat*}{2}
d^v_{00}&=\begin{pmatrix}
\diffd z_1\\
\diffd z_2
\end{pmatrix},
&
d^v_{01}&=\begin{pmatrix}
-\diffd z_2 & \diffd z_1
\end{pmatrix},\\
d^v_{10}&=\begin{pmatrix}
\difl z_1 & 0 \\
\difl z_2 & 0 \\
0 & \difr z_1 \\
0 & \difr z_2 
\end{pmatrix},
&\quad
d^v_{11}&=\begin{pmatrix}
-\difl z_2 & \difl z_1 & 0 & 0 \\
0 & 0 & -\difr z_2 & \difr z_1 
\end{pmatrix},\\
d^v_{20}&=\begin{pmatrix}
\diffd z_1 & 0 \\
\diffd z_2 & 0 \\
0 & \difb z_1 \\
0 & \difb z_2 
\end{pmatrix},
&
d^v_{21}&=\begin{pmatrix}
-\diffd z_2 & \diffd z_1 & 0 & 0 \\
0 & 0 & -\difb z_2 & \difb z_1 
\end{pmatrix},
\end{alignat*}
and the rest are hence known according to the alternating feature. Next we add appropriate morphisms $d^h$, $t$ making $\mc P_{\sbu\sbu}$ into a homotopy double complex. The morphisms are given as follows:
\begin{align*}
d^h_{00}&=\begin{pmatrix}
1\ot x-x\ot 1 \\
1\ot y-y\ot 1
\end{pmatrix},
\\
d^h_{01}&=\begin{pmatrix}
x\ot 1-(1\ot x)\De_1(f_1) & -(1\ot x)\De_2(f_1) \\
-(1\ot x)\De_1(f_2) & x\ot 1-(1\ot x)\De_2(f_2) \\
-1\ot y+(y\ot 1)\De_1(f_1) & (y\ot 1)\De_2(f_1) \\
(y\ot 1)\De_1(f_2) & -1\ot y+(y\ot 1)\De_2(f_2)
\end{pmatrix},
\\
d^h_{02}&=\begin{pmatrix}
-x\ot 1+(1\ot x)J_\nc \\
1\ot y-(y\ot 1)J_\nc
\end{pmatrix},
\\
d^h_{10}&=\begin{pmatrix}
y\ot 1 & 1\ot x \\
1\ot y & x\ot 1
\end{pmatrix},
\\
d^h_{11}&=\begin{pmatrix}
-y\ot 1 & 0 & -1\ot x & 0 \\
0 & -y\ot 1 & 0 & -1\ot x \\
-1\ot y & 0 & -x\ot 1 & 0 \\
0 & -1\ot y & 0 & -x\ot 1
\end{pmatrix},
\\
d^h_{12}&=\begin{pmatrix}
y\ot 1 & 1\ot x \\
1\ot y & x\ot 1
\end{pmatrix},
\\
d^h_{20}&=\begin{pmatrix}
-x\ot 1 & 1\ot x \\
1\ot y & -y\ot 1
\end{pmatrix},
\\
d^h_{21}&=\begin{pmatrix}
x\ot 1 & 0 & -1\ot x & 0 \\
0 &  x\ot 1 & 0 & -1\ot x \\
-1\ot y & 0 & y\ot 1 & 0 \\
0 & -1\ot y & 0 & y\ot 1
\end{pmatrix},
\\
d^h_{22}&=\begin{pmatrix}
-x\ot 1 & 1\ot x \\
1\ot y & -y\ot 1
\end{pmatrix},
\\
t_{01}&=\begin{pmatrix}
\De_1(\vphi) & \De_2(\vphi) \\
\Delb_1(\vphi)\De_1(f_1)+\Delb_2(\vphi)\De_1(f_2) & \Delb_1(\vphi)\De_2(f_1)+\Delb_2(\vphi)\De_2(f_2)
\end{pmatrix},
\\
t_{02}&=\begin{pmatrix}
-\De_2(\vphi) \\
\De_1(\vphi) \\
-J_\nc\Delb_2(\vphi) \\
J_\nc\Delb_1(\vphi)
\end{pmatrix},
\\
t_{11}&=\begin{pmatrix}
\Dell_1(\vphi) & \Dell_2(\vphi) & 0 & 0 \\
0 & 0 & \Delr_1(\vphi) & \Delr_2(\vphi)
\end{pmatrix},
\\
t_{12}&=\begin{pmatrix}
-\Dell_2(\vphi) & 0 \\
\Dell_1(\vphi) & 0 \\
0 & -\Delr_2(\vphi) \\
0 & \Delr_1(\vphi)
\end{pmatrix},
\\
t_{21}&=\begin{pmatrix}
\De_1(\vphi) & \De_2(\vphi) & 0 & 0 \\
0 & 0 & \Delb_1(\vphi) & \Delb_2(\vphi)
\end{pmatrix},
\\
t_{22}&=\begin{pmatrix}
-\De_2(\vphi) & 0 \\
\De_1(\vphi) & 0 \\
0 & -\Delb_2(\vphi) \\
0 & \Delb_1(\vphi)
\end{pmatrix}.
\end{align*}

\begin{proposition}
	The $4$-tuple $(\mc P_{\sbu\sbu},d^v,d^h,t)$ is a homotopy double cochain complex.
\end{proposition}

\begin{proof}
	Notice that \eqref{eq:htpy-def-1} is clearly satisfied, and so \eqref{eq:htpy-def-2}-\eqref{eq:htpy-def-5} are to be verified. All the verifications are translated into anti-multiplications of matrices over $W_{(2)}^e$.  We confine ourselves to the proof of
	\begin{gather}
	d^h_{00}d^h_{10}+d^v_{00}t_{01}=0, \label{eq:htpy-proof-1}\\
	d^h_{01}d^h_{11}+d^v_{01}t_{02}+t_{01}d^v_{20}=0, \label{eq:htpy-proof-2}
	\end{gather}
	leaving others to the reader.
	
	We have
	\[
	d^h_{00}d^h_{10}=\begin{pmatrix}
	y\ot 1 & 1\ot x \\
	1\ot y & x\ot 1
	\end{pmatrix}\begin{pmatrix}
	1\ot x-x\ot 1 \\
	1\ot y-y\ot 1
	\end{pmatrix}
	=\begin{pmatrix}
	1\ot \vphi -\vphi\ot 1\\
	1\ot \si(\vphi)-\si(\vphi)\ot 1
	\end{pmatrix},
	\]
	and $d^v_{00}t_{01}$ equals
	\begin{align*}
	&\varphantom{=}\begin{pmatrix}
	\De_1(\vphi) & \De_2(\vphi) \\
	\Delb_1(\vphi)\De_1(f_1)+\Delb_2(\vphi)\De_1(f_2) & \Delb_1(\vphi)\De_2(f_1)+\Delb_2(\vphi)\De_2(f_2)
	\end{pmatrix}\begin{pmatrix}
	\diffd z_1\\
	\diffd z_2
	\end{pmatrix} \\
	&=\begin{pmatrix}
	\De_1(\vphi)\diffd z_1+\De_2(\vphi)\diffd z_2 \\
	(\Delb_1(\vphi)\De_1(f_1)+\Delb_2(\vphi)\De_1(f_2))\diffd z_1+(\Delb_1(\vphi)\De_2(f_1)+\Delb_2(\vphi)\De_2(f_2))\diffd z_2
	\end{pmatrix} \\
	&=\begin{pmatrix}
	\vphi\ot 1-1\ot \vphi \\
	\Delb_1(\vphi)\difb z_1+\Delb_2(\vphi)\difb z_2
	\end{pmatrix} \qquad(\text{by Lemma \ref{lem:NC-diff}}) \\
	&=\begin{pmatrix}
	\vphi\ot 1-1\ot \vphi \\
	\si(\vphi)\ot 1-1\ot\si(\vphi)
	\end{pmatrix}.
	\end{align*}
	Hence \eqref{eq:htpy-proof-1} is proven.
	
	For \eqref{eq:htpy-proof-2}, we have
	\begin{align*}
	d^h_{01}d^h_{11}&=\begin{pmatrix}
	-y\ot 1 & 0 & -1\ot x & 0 \\
	0 & -y\ot 1 & 0 & -1\ot x \\
	-1\ot y & 0 & -x\ot 1 & 0 \\
	0 & -1\ot y & 0 & -x\ot 1
	\end{pmatrix}\\
	&\varphantom{=}{}\cdot\begin{pmatrix}
	x\ot 1-(1\ot x)\De_1(f_1) & -(1\ot x)\De_2(f_1) \\
	-(1\ot x)\De_1(f_2) & x\ot 1-(1\ot x)\De_2(f_2) \\
	-1\ot y+(y\ot 1)\De_1(f_1) & (y\ot 1)\De_2(f_1) \\
	(y\ot 1)\De_1(f_2) & -1\ot y+(y\ot 1)\De_2(f_2)
	\end{pmatrix} \\
	&=\begin{pmatrix}
	-yx\ot 1+1\ot yx & 0 \\ 0 & -yx\ot 1+1\ot yx \\ (1\ot xy-xy\ot 1)\De_1(f_1) & (1\ot xy-xy\ot 1)\De_2(f_1) \\ (1\ot xy-xy\ot 1)\De_1(f_2) & (1\ot xy-xy\ot 1)\De_2(f_2)
	\end{pmatrix} \\
	&=\begin{pmatrix}
	1\ot \vphi-\vphi\ot 1 & 0 \\ 0 & 1\ot \vphi-\vphi\ot 1 \\ (1\ot \si(\vphi)-\si(\vphi)\ot 1)\De_1(f_1) & (1\ot \si(\vphi)-\si(\vphi)\ot 1)\De_2(f_1) \\ (1\ot \si(\vphi)-\si(\vphi)\ot 1)\De_1(f_2) & (1\ot \si(\vphi)-\si(\vphi)\ot 1)\De_2(f_2)
	\end{pmatrix}, \\
	d^v_{01}t_{02}&=\begin{pmatrix}
	-\De_2(\vphi) \\
	\De_1(\vphi) \\
	-J\Delb_2(\vphi) \\
	J\Delb_1(\vphi)
	\end{pmatrix}\begin{pmatrix}
	-\diffd z_2 & \diffd z_1
	\end{pmatrix}=\begin{pmatrix}
		\De_2(\vphi)\diffd z_2 & -\De_2(\vphi)\diffd z_1 \\
		-\De_1(\vphi)\diffd z_2 & \De_1(\vphi)\diffd z_1 \\
		J_\nc\Delb_2(\vphi)\diffd z_2 & -J_\nc\Delb_2(\vphi)\diffd z_1 \\
		-J_\nc\Delb_1(\vphi)\diffd z_2 & J_\nc\Delb_1(\vphi)\diffd z_1
		\end{pmatrix},
	\end{align*}
	and $t_{01}d^v_{20}$ is equal to
	\[
		\begin{pmatrix}
				\De_1(\vphi)\diffd z_1 & \De_2(\vphi)\diffd z_1 \\
				\De_1(\vphi)\diffd z_2 & \De_2(\vphi)\diffd z_2 \\
				(\Delb_1(\vphi)\De_1(f_1)+\Delb_2(\vphi)\De_1(f_2))\difb z_1 & (\Delb_1(\vphi)\De_2(f_1)+\Delb_2(\vphi)\De_2(f_2))\difb z_1 \\
				(\Delb_1(\vphi)\De_1(f_1)+\Delb_2(\vphi)\De_1(f_2))\difb z_2 & (\Delb_1(\vphi)\De_2(f_1)+\Delb_2(\vphi)\De_2(f_2))\difb z_2 
				\end{pmatrix}\!.
	\]
	It follows that the $(1,1)$-, $(1,2)$-, $(2,1)$-, $(2,2)$-entries of $d^h_{01}d^h_{11}+d^v_{01}t_{02}+t_{01}d^v_{20}$ are all zero. The $(3,1)$-entry is
	\begin{align*}
	&\varphantom{=}-\difb\vphi\,\De_1(f_1)+J_\nc\Delb_2(\vphi)\diffd z_2+(\Delb_1(\vphi)\De_1(f_1)+\Delb_2(\vphi)\De_1(f_2))\difb z_1 \\
	&=-\difb\vphi\,\De_1(f_1)+J_{\nc}\Delb_2(\vphi)\diffd z_2+\Delb_1(\vphi)\De_1(f_1)\difb z_1+\Delb_2(\vphi)\De_1(f_2)\difb z_1\\
	&=(\Delb_1(\vphi)\difb z_1-\difb\vphi)\De_1(f_1)+J_{\nc}\diffd z_2+\De_1(f_2)\difb z_1)\Delb_2(\vphi)\\
	&=(-\Delb_2(\vphi)\difb z_2)\De_1(f_1)+J_{\nc}\diffd z_2+\De_1(f_2)\difb z_1)\Delb_2(\vphi)\\
	&=(J_{\nc}\diffd z_2-\De_1(f_1)\difb z_2+\De_1(f_2)\difb z_1)\Delb_2(\vphi)\\
	&=(J_{\nc}\diffd z_2-\De_1(f_1)(\De_1(f_2)\diffd z_1+\De_2(f_2)\diffd z_2)+\De_1(f_2)(\De_1(f_1)\diffd z_1\\
	&\varphantom{=}{}+\De_2(f_1)\diffd z_2))\Delb_2(\vphi)\\
	&=(J_{\nc}-\De_1(f_1)\De_2(f_2)+\De_1(f_2)\De_2(f_1))\diffd z_2\Delb_2(\vphi)\\
	&=0.
	\end{align*}
	Similarly, the $(3,2)$-, $(4,1)$-, $(4,2)$-entries are zero. Thus \eqref{eq:htpy-proof-2} is also proven.
\end{proof}

\begin{remark}
	When constructing homotopy double complex for $W_{(1)}$ in \cite{L:gwa-def}, the verification of \eqref{eq:htpy-def-4}, \eqref{eq:htpy-def-5} is trivial. But for $W_{(2)}$, this is not so easy.
\end{remark}

\begin{proposition}\label{prop:free-resol-GWA}
If $\vphi\neq0$, $\Tot \mc P_{\sbu\sbu}$ is a resolution of $W_{(2)}$ by finitely generated free $W_{(2)}^e$-modules via $\mu$. Moreover, the complex $\Tot \mc P_{\sbu\sbu}$ is alternate in $[3, +\infty)$.
\end{proposition}

\begin{proof}
This follows from spectral sequence argument. See \cite{L:gwa-def} for the details.
\end{proof}

\subsection{Proof of sufficiency}
Suppose $(\vphi,\vphi_1,\vphi_2)=B$. Let us prove that $W_{(2)}$ is homologically smooth in this case. 

First of all, notice that $\vphi\neq 0$ is automatically satisfied. Hence by Proposition \ref{prop:free-resol-GWA}, $\Tot\mc P_{\sbu\sbu}$ is a free resolution of $W_{(2)}$, and so we can compute Hochschild cohomology $H^*(W_{(2)},M)=\Ext^*_{W_{(2)}^e}(W_{(2)},M)$ for any $W_{(2)}^e$-module $M$ by $\Tot\mc P_{\sbu\sbu}$. Next let $\mc Q_M^{\sbu\sbu}=\Hom_{W_{(2)}^e}(\mc P_{\sbu\sbu}, M)$, $d_v^{pq}=\Hom_{W_{(2)}^e}(d^v_{pq}, M)$, $d_h^{pq}=\Hom_{W_{(2)}^e}(d^h_{pq}, M)$, and $t^{pq}=\Hom_{W_{(2)}^e}(t_{pq}, M)$. Clearly, $(\mc Q_M^{\sbu\sbu},d_v,d_h,t)$ is a homotopy double cochain complex, and
\[
H^*(\Tot\mc Q_M^{\sbu\sbu})\cong H^*(\Hom_{W^e}(\Tot\mc P_{\sbu\sbu}, M))\cong H^*(W_{(2)},M).
\]
We write $\mc{Q}_M^{\sbu}$ schematically, as follows.
\[
\xymatrix@R=12mm@C=12mm{
	M \ar[r]^-{d_h^{02}}\ar[rrd]^(0.6){t^{02}}|!{[r];[rd]}\hole  & M^2 \ar[r]^-{d_h^{12}}\ar[rrd]^(0.6){t^{12}}|!{[r];[rd]}\hole & M^2 \ar[r]^-{d_h^{22}}\ar[rrd]^(0.6){t^{22}}|!{[r];[rd]}\hole & M^2 \ar[r]^-{d_h^{32}} & M^2 \ar[r] & \cdots \\
	M^2 \ar[r]^-{d_h^{01}}\ar[rrd]^(0.6){t^{01}}|!{[r];[rd]}\hole\ar[u]^(0.4){d_v^{01}}  & M^4 \ar[r]^-{d_h^{11}}\ar[rrd]^(0.6){t^{11}}|!{[r];[rd]}\hole\ar[u]^(0.4){d_v^{11}} & M^4 \ar[r]^-{d_h^{21}}\ar[rrd]^(0.6){t^{21}}|!{[r];[rd]}\hole\ar[u]^(0.4){d_v^{21}} & M^4 \ar[r]^-{d_h^{31}}\ar[u]^(0.4){d_v^{31}} & M^4 \ar[r]\ar[u]^(0.4){d_v^{41}} & \cdots \\
	M \ar[r]_-{d_h^{00}}\ar[u]^(0.4){d_v^{00}}  & M^2 \ar[r]_-{d_h^{10}}\ar[u]^(0.4){d_v^{10}} & M^2 \ar[r]_-{d_h^{20}}\ar[u]^(0.4){d_v^{20}} & M^2 \ar[r]_-{d_h^{30}}\ar[u]^(0.4){d_v^{30}} & M^2 \ar[r]\ar[u]^(0.4){d_v^{40}} & \cdots
}
\]

\begin{proposition}\label{prop:smooth-condition}
One has $H^4(\Tot\mc Q_M^{\sbu\sbu})=0$ for all $M$ if $(\vphi,\vphi_1,\vphi_2)=B$. Consequently, $W_{(2)}$ is homologically smooth if $(\vphi,\vphi_1,\vphi_2)=B$.
\end{proposition}

\begin{proof}
There exist polynomials $\al$, $\be_1$, $\be_2$ such that $\al\vphi+\be_1\vphi_1+\be_2\vphi_2=1$. We write $d_{\mc Q}^\sbu$ for the differentials of $\Tot\mc Q_M^{\sbu\sbu}$.

Let $\mbf{m}=(m^{22}_1, m^{22}_2, m^{31}_1, m^{31}_2, m^{31}_3, m^{31}_4, m^{40}_1, m^{40}_2)$ be a $4$-cocycle of $\Tot\mc Q_M^{\sbu\sbu}$. We have
\[
\left\{
\begin{aligned}
d_h^{22}(m^{22}_1,m^{22}_2)+d_v^{31}(m^{31}_1, m^{31}_2, m^{31}_3, m^{31}_4)&=0 \\
t^{22}(m^{22}_1,m^{22}_2)+d_h^{31}(m^{31}_1, m^{31}_2, m^{31}_3, m^{31}_4)+d_v^{40}(m^{40}_1,m^{40}_2)&=0 \\
t^{31}(m^{31}_1, m^{31}_2, m^{31}_3, m^{31}_4)+d_h^{40}(m^{40}_1,m^{40}_2)&=0
\end{aligned}
\right.
\]
which is equivalent to the following eight equations
\begin{align}
-(x\ot 1)m^{22}_1+(1\ot x)m^{22}_2-\difl z_2 m^{31}_1+\difl z_1 m^{31}_2&=0, \label{eq:cocyc-1}\\
(1\ot y)m^{22}_1-(y\ot 1)m^{22}_2-\difr z_2m^{31}_3+\difr z_1m^{31}_4&=0, \label{eq:cocyc-2}\\
-\De_2(\vphi)m^{22}_1-(y\ot 1)m^{31}_1-(1\ot x)m^{31}_3+\diffd z_1m^{40}_1&=0, \label{eq:cocyc-3}\\
\De_1(\vphi)m^{22}_1-(y\ot 1)m^{31}_2-(1\ot x)m^{31}_4+\diffd z_2m^{40}_1&=0, \label{eq:cocyc-4}\\
-\Delb_2(\vphi)m^{22}_2-(1\ot y)m^{31}_1-(x\ot 1)m^{31}_3+\difb z_1m^{40}_2&=0, \label{eq:cocyc-5}\\
\Delb_1(\vphi)m^{22}_2-(1\ot y)m^{31}_2-(x\ot 1)m^{31}_4+\difb z_2m^{40}_2&=0, \label{eq:cocyc-6}\\
\Dell_1(\vphi)m^{31}_1+\Dell_2(\vphi)m^{31}_2-(x\ot 1)m^{40}_1+(1\ot x)m^{40}_2&=0, \label{eq:cocyc-7}\\
\Delr_1(\vphi)m^{31}_3+\Delr_2(\vphi)m^{31}_4+(1\ot y)m^{40}_1-(y\ot 1)m^{40}_2&=0.\label{eq:cocyc-8}
\end{align}

Define
\begin{align*}
n^{12}_1&=(1\ot \be_1)m^{31}_2-(1\ot \be_2)m^{31}_1, \\
n^{12}_2&=(1\ot \si(\be_1))m^{31}_4-(1\ot \si(\be_2))m^{31}_3
+(1\ot \al y)m^{22}_1, \\
n^{21}_1&=(1\ot \be_1)m^{40}_1+(1\ot \be_2)\De_2^{\pl_2}(\vphi)m^{22}_1, \\
n^{21}_2&=(1\ot \be_2)m^{40}_1-(1\ot \be_1)\De_1^{\pl_1}(\vphi)m^{22}_1, \\
n^{21}_3&=(1\ot \si(\be_1))m^{40}_2-(1\ot \al y)m^{31}_1+(1\ot \si(\be_2))\Dell_2^{\si\pl_2}(\vphi)m^{22}_2, \\
n^{21}_4&=(1\ot \si(\be_2))m^{40}_2-(1\ot \al y)m^{31}_2-(1\ot \si(\be_1))\Dell_1^{\si\pl_1}(\vphi)m^{22}_2, \\
n^{30}_1&=-(1\ot \be_1)\Dell_1^{\pl_1}(\vphi)m^{31}_1-(1\ot\be_2)\Dell_2^{\pl_2}(\vphi)m^{31}_2, \\
n^{30}_2&=(1\ot \al y)m^{40}_1-(1\ot \si(\be_1))\De_1^{\si\pl_1}(\vphi)m^{31}_3-(1\ot \si(\be_2))\De_2^{\si\pl_2}(\vphi)m^{31}_4.
\end{align*}
These $n^{**}_{*}$ constitute a $3$-cochain $\mbf{n}$. Let us prove $d_{\mc Q}^3(\mbf{n})=\mbf{m}$. The following three equations are to be checked:
\begin{equation}
\label{eq:m-n-1}
\begin{pmatrix}
m^{22}_1 \\ m^{22}_2
\end{pmatrix}=
\begin{pmatrix}
y\ot 1 & 1\ot x \\ 1\ot y & x\ot 1
\end{pmatrix}
\begin{pmatrix}
n^{12}_1 \\ n^{12}_2
\end{pmatrix}+
\begin{pmatrix}
-\diffd z_2 & \diffd z_1 & 0 & 0 \\ 0 & 0 & -\difb z_2 & \difb z_1
\end{pmatrix}
\begin{pmatrix}
n^{21}_1 \\ n^{21}_2 \\ n^{21}_3 \\ n^{21}_4
\end{pmatrix},
\end{equation}
\begin{equation}\label{eq:m-n-2}
\begin{split}
\begin{pmatrix}
m^{31}_1 \\ m^{31}_2 \\ m^{31}_3 \\ m^{31}_4
\end{pmatrix}&=
\begin{pmatrix}
-\Dell_2(\vphi) & 0 \\ \Dell_1(\vphi) & 0 \\ 0 & \!-\Delr_2(\vphi) \\ 0 & \Delr_1(\vphi)
\end{pmatrix}
\begin{pmatrix}
n^{12}_1 \\ n^{12}_2
\end{pmatrix}+
\begin{pmatrix}
x\ot 1 & 0 & -1\ot x & 0 \\ 0 & x\ot 1 & 0 & -1\ot x \\ -1\ot y & 0 & y\ot 1 & 0 \\ 0 & -1\ot y & 0 & y\ot 1
\end{pmatrix}\\
&\varphantom{=}{}\cdot
\begin{pmatrix}
n^{21}_1 \\ n^{21}_2 \\ n^{21}_3 \\ n^{21}_4
\end{pmatrix}+\begin{pmatrix}
\difl z_1 & 0 \\ \difl z_2 & 0 \\ 0 & \difr z_1 \\ 0 & \difr z_2
\end{pmatrix}
\begin{pmatrix}
n^{30}_1 \\ n^{30}_2
\end{pmatrix},
\end{split}
\end{equation}
\begin{equation}\label{eq:m-n-3}
\begin{pmatrix}
m^{40}_1 \\ m^{40}_2
\end{pmatrix}=
\begin{pmatrix}
\Dell_1(\vphi) & \Dell_2(\vphi) & 0 & 0 \\ 0 & 0 & \!\Delb_1(\vphi) & \!\Delb_2(\vphi)
\end{pmatrix}
\begin{pmatrix}
n^{21}_1 \\ n^{21}_2 \\ n^{21}_3 \\ n^{21}_4
\end{pmatrix}+
\begin{pmatrix}
y\ot 1 & 1\ot x \\ 1\ot y & x\ot 1
\end{pmatrix}
\begin{pmatrix}
n^{30}_1 \\ n^{30}_2
\end{pmatrix}.
\end{equation}

There are eight equalities in total to be verified and the verification is tediously long. So we divide the whole proof into four lemmas. The sufficiency follows from them.
\end{proof}

\begin{lemma}\label{lem:Del-partial}
	For $i=1$, $2$, we have
	\begin{enumerate}
		\item $\De_i(\vphi)-\De_i^{\pl_i}(\vphi)\diffd z_i=1\ot\vphi_i$,
		\item $\Dell_i(\vphi)-\Dell_i^{\pl_i}(\vphi)\difl z_i=1\ot\vphi_i$,
		\item $\Delr_i(\vphi)-\De_i^{\si\pl_i}(\vphi)\difr z_i=1\ot\si(\vphi_i)$,
		\item $\Delb_i(\vphi)-\Dell_i^{\si\pl_i}(\vphi)\difb z_i=1\ot\si(\vphi_i)$.
	\end{enumerate}
\end{lemma}

\begin{proof}
	Clearly, (2), (3), (4) follow from (1). So let us prove (1) for $i=1$. The case $i=2$ is left to the reader.
	
	Suppose $\vphi=\sum_{i,j}\lam_{ij}z_1^iz_2^j$. Then
	\[
	\De_1(\vphi)=\sum_{i,j}\sum_{k=1}^{i}\lam_{ij}z_1^{i-k}\ot z_1^{k-1}z_2^j
	\]
	and so
	\begin{align*}
	\De_1^{\pl_1}(\vphi)\diffd z_1&=\sum_{i,j}\sum_{k=1}^{i}\lam_{ij}z_1^{i-k}\ot \pl_1(z_1^{k-1}z_2^j)(z_1\ot 1-1\ot z_1) \\
	&=\sum_{i,j}\sum_{k=1}^{i}\lam_{ij}z_1^{i-k}\ot (k-1)z_1^{k-2}z_2^j(z_1\ot 1-1\ot z_1) \\
	&=\sum_{i,j}\sum_{k=1}^{i}\lam_{ij}(k-1)z_1^{i-k+1}\ot z_1^{k-2}z_2^j-\sum_{i,j}\sum_{k=1}^{i}\lam_{ij}(k-1)z_1^{i-k}\ot z_1^{k-1}z_2^j \\
	&=\sum_{i,j}\sum_{k=0}^{i-1}\lam_{ij}kz_1^{i-k}\ot z_1^{k-1}z_2^j-\sum_{i,j}\sum_{k=1}^{i}\lam_{ij}(k-1)z_1^{i-k}\ot z_1^{k-1}z_2^j \\
	&=\sum_{i,j}\sum_{k=1}^{i}\lam_{ij}z_1^{i-k}\ot z_1^{k-1}z_2^j-\sum_{i,j}\lam_{ij}i\ot z_1^{i-1}z_2^j \\
	&=\De_1(\vphi)-1\ot\vphi_1. \qedhere
	\end{align*}
\end{proof}

\begin{lemma}
	Eq.\ \eqref{eq:m-n-1} holds true.
\end{lemma}

\begin{proof}
	We have
	\begin{align*}
	&\varphantom{=}(y\ot 1)n^{12}_1+(1\ot x)n^{12}_2-\diffd z_2n^{21}_1+\diffd z_1n^{21}_2\\
	&=(y\ot 1)[(1\ot \be_1)m^{31}_2-(1\ot \be_2)m^{31}_1]+(1\ot x)[(1\ot \si(\be_1))m^{31}_4-(1\ot \si(\be_2))m^{31}_3\\
	&\varphantom{=}{}+(1\ot \al y)m^{22}_1]-\diffd z_2[(1\ot \be_1)m^{40}_1+(1\ot \be_2)\De_2^{\pl_2}(\vphi)m^{22}_1]+\diffd z_1[(1\ot \be_2)m^{40}_1\\
	&\varphantom{=}{}-(1\ot \be_1)\De_1^{\pl_1}(\vphi)m^{22}_1]\\
	&=(1\ot \be_1)[(y\ot 1)m^{31}_2+(1\ot x)m^{31}_4-\diffd z_2m^{40}_1-\De_1^{\pl_1}(\vphi)\diffd z_1m^{22}_1]+(1\ot \be_2)\\
	&\varphantom{=}{}\cdot[-(y\ot 1)m^{31}_1-(1\ot x)m^{31}_3-\De_2^{\pl_2}(\vphi)\diffd z_2m^{22}_1+\diffd z_1m^{40}_1]+(1\ot\al yx)m^{22}_1\\
	&\overset{\dagger_1}{=}(1\ot \be_1)[\De_1(\vphi)m^{22}_1-\De_1^{\pl_1}(\vphi)\diffd z_1m^{22}_1]+(1\ot \be_2)[\De_2(\vphi)m^{22}_1-\De_2^{\pl_2}(\vphi)\diffd z_2m^{22}_1]\\
	&\varphantom{=}{}+(1\ot\al \vphi)m^{22}_1\\
	&\overset{\dagger_2}{=}(1\ot \be_1)(1\ot\vphi_1)m^{22}_1+(1\ot \be_2)(1\ot\vphi_2)m^{22}_1+(1\ot\al \vphi)m^{22}_1\\
	&=m^{22}_1
	\end{align*}
	where $\dagger_1$ follows from \eqref{eq:cocyc-4}, \eqref{eq:cocyc-3} , and $\dagger_2$ from Lemma \ref{lem:Del-partial} (1). Also,
	\begin{align*}
	&\varphantom{=}(1\ot y)n^{12}_1+(x\ot 1)n^{12}_2-\difb z_2n^{21}_3+\difb z_1n^{21}_4\\
	&=(1\ot y)[(1\ot \be_1)m^{31}_2-(1\ot \be_2)m^{31}_1]+(x\ot 1)[(1\ot \si(\be_1))m^{31}_4-(1\ot \si(\be_2))m^{31}_3
	\\
	&\varphantom{=}{}+(1\ot \al y)m^{22}_1]-\difb z_2[(1\ot \si(\be_1))m^{40}_2-(1\ot \al y)m^{31}_1+(1\ot \si(\be_2))\\
	&\varphantom{=}{}\cdot\Dell_2^{\si\pl_2}(\vphi)m^{22}_2]+\difb z_1[(1\ot \si(\be_2))m^{40}_2-(1\ot \al y)m^{31}_2-(1\ot \si(\be_1))\\
	&\varphantom{=}{}\cdot\Dell_1^{\si\pl_1}(\vphi)m^{22}_2]\\
	&=(1\ot \si(\be_1))[(1\ot y)m^{31}_2+(x\ot 1)m^{31}_4-\difb z_2m^{40}_2-\Dell_1^{\si\pl_1}(\vphi)\difb z_1m^{22}_2]\\
	&\varphantom{=}{}+(1\ot \si(\be_2))[-(1\ot y)m^{31}_1-(x\ot 1)m^{31}_3-\Dell_2^{\si\pl_2}(\vphi)\difb z_2m^{22}_2+\difb z_1m^{40}_2]\\
	&\varphantom{=}{}+(1\ot\al y)[(x\ot 1)m^{22}_1+\difl z_2m^{31}_1-\difl z_1m^{31}_2]\\
	&\overset{\dagger_3}{=}(1\ot \si(\be_1))[\Delb_1(\vphi)m^{22}_2-\Dell_1^{\si\pl_1}(\vphi)\difb z_1m^{22}_2]+(1\ot \si(\be_2))[\Delb_2(\vphi)m^{22}_2\\
	&\varphantom{=}{}-\Dell_2^{\si\pl_2}(\vphi)\difb z_2m^{22}_2]+(1\ot\al y)(1\ot x)m^{22}_2\\
	&\overset{\dagger_4}{=}(1\ot \si(\be_1))(1\ot\si(\vphi_1))m^{22}_2+(1\ot \si(\be_2))(1\ot\si(\vphi_2))m^{22}_2+(1\ot\si(\al \vphi))m^{22}_2\\
	&=m^{22}_2
	\end{align*}
	where $\dagger_3$ follows from \eqref{eq:cocyc-6}, \eqref{eq:cocyc-5}, \eqref{eq:cocyc-1}, and $\dagger_4$ from Lemma \ref{lem:Del-partial} (4).
\end{proof}

\begin{lemma}
	Eq.\ \eqref{eq:m-n-2} holds true.
\end{lemma}

\begin{proof}
	We have
	\begin{align*}
	&\varphantom{=}-\Dell_2(\vphi)n^{12}_1+(x\ot 1)n^{21}_1-(1\ot x)n^{21}_3+\difl z_1n^{30}_1\\
	&=-\Dell_2(\vphi)[(1\ot \be_1)m^{31}_2-(1\ot \be_2)m^{31}_1]+(x\ot 1)[(1\ot \be_1)m^{40}_1\\
	&\varphantom{=}{}+(1\ot \be_2)\De_2^{\pl_2}(\vphi)m^{22}_1]-(1\ot x)[(1\ot \si(\be_1))m^{40}_2-(1\ot \al y)m^{31}_1+(1\ot \si(\be_2))\\
	&\varphantom{=}{}\cdot\Dell_2^{\si\pl_2}(\vphi)m^{22}_2]+\difl z_1[-(1\ot \be_1)\Dell_1^{\pl_1}(\vphi)m^{31}_1-(1\ot\be_2)\Dell_2^{\pl_2}(\vphi)m^{31}_2]\\
	&=(1\ot \be_1)[-\Dell_2(\vphi)m^{31}_2+(x\ot 1)m^{40}_1-(1\ot x)m^{40}_2-\Dell_1^{\pl_1}(\vphi)\difl z_1m^{31}_1]\\
	&\varphantom{=}{}+(1\ot \be_2)[\Dell_2(\vphi)m^{31}_1+\Dell_2^{\pl_2}(\vphi)(x\ot 1)m^{22}_1-\Dell_2^{\pl_2}(\vphi)(1\ot x)m^{22}_2\\
	&\varphantom{=}{}-\Dell_2^{\pl_2}(\vphi)\difl z_1m^{31}_2]+(1\ot\al yx)m^{31}_1\\
	&\overset{\dagger_1}{=}(1\ot \be_1)[\Dell_1(\vphi)m^{31}_1-\Dell_1^{\pl_1}(\vphi)\difl z_1m^{31}_1]+(1\ot \be_2)[\Dell_2(\vphi)m^{31}_1\\
	&\varphantom{=}{}-\Dell_2^{\pl_2}(\vphi)\difl z_2m^{31}_1]\\
	&\varphantom{=}{}+(1\ot\al \vphi)m^{31}_1\\
	&\overset{\dagger_2}{=}(1\ot \be_1)(1\ot\vphi_1)m^{31}_1+(1\ot \be_2)(1\ot\vphi_2)m^{31}_1+(1\ot\al \vphi)m^{31}_1\\
	&=m^{31}_1
	\end{align*}
	where $\dagger_1$ follows from \eqref{eq:cocyc-7}, \eqref{eq:cocyc-1}, and $\dagger_2$ from Lemma \ref{lem:Del-partial} (2). Next we have
	\begin{align*}
	&\varphantom{=}\Dell_1(\vphi)n^{12}_1+(x\ot 1)n^{21}_2-(1\ot x)n^{21}_4+\difl z_2n^{30}_1\\
	&=\Dell_1(\vphi)[(1\ot \be_1)m^{31}_2\!-\!(1\ot \be_2)m^{31}_1]+(x\ot 1)[(1\ot \be_2)m^{40}_1\!-\!(1\ot \be_1)\De_1^{\pl_1}(\vphi)m^{22}_1]\\
	&\varphantom{=}{}-(1\ot x)[(1\ot \si(\be_2))m^{40}_2-(1\ot \al y)m^{31}_2-(1\ot \si(\be_1))\Dell_1^{\si\pl_1}(\vphi)m^{22}_2]\\
	&\varphantom{=}{}+\difl z_2[-(1\ot \be_1)\Dell_1^{\pl_1}(\vphi)m^{31}_1-(1\ot\be_2)\Dell_2^{\pl_2}(\vphi)m^{31}_2]\\
	&=(1\ot \be_1)[\Dell_1(\vphi)m^{31}_2-\Dell_1^{\pl_1}(\vphi)(x\ot 1)m^{22}_1+\Dell_1^{\pl_1}(\vphi)(1\ot x)m^{22}_2-\Dell_1^{\pl_1}(\vphi)\\
	&\varphantom{=}{}\cdot\difl z_2m^{31}_1]+(1\ot \be_2)[-\Dell_1(\vphi)m^{31}_1+(x\ot 1)m^{40}_1-(1\ot x)m^{40}_2-\Dell_2^{\pl_2}(\vphi)\\
	&\varphantom{=}{}\cdot\difl z_2m^{31}_2]+(1\ot\al yx)m^{31}_2\\
	&\overset{\dagger_3}{=}(1\ot \be_1)[\Dell_1(\vphi)m^{31}_2\!-\!\Dell_1^{\pl_1}(\vphi)\difl z_1m^{31}_2]\!+\!(1\ot \be_2)[\Dell_2(\vphi)m^{31}_2\!-\!\Dell_2^{\pl_2}(\vphi)\difl z_2m^{31}_2]\\
	&\varphantom{=}{}+(1\ot\al \vphi)m^{31}_2\\
	&\overset{\dagger_4}{=}(1\ot \be_1)(1\ot\vphi_1)m^{31}_2+(1\ot \be_2)(1\ot\vphi_2)m^{31}_2+(1\ot\al \vphi)m^{31}_2\\
	&=m^{31}_2
	\end{align*}
	where $\dagger_3$, $\dagger_4$ again follow from \eqref{eq:cocyc-1}, \eqref{eq:cocyc-7}, and Lemma \ref{lem:Del-partial} (2) respectively. For the third,
	\begin{align*}
	&\varphantom{=}-\Delr_2(\vphi)n^{12}_2-(1\ot y)n^{21}_1+(y\ot 1)n^{21}_3+\difr z_1n^{30}_2\\
	&=-\Delr_2(\vphi)[(1\ot \si(\be_1))m^{31}_4\!-(1\ot \si(\be_2))m^{31}_3
	\!+(1\ot \al y)m^{22}_1]-(1\ot y)[(1\ot \be_1)m^{40}_1\\
	&\varphantom{=}{}+(1\ot \be_2)\De_2^{\pl_2}(\vphi)m^{22}_1]+(y\ot 1)[(1\ot \si(\be_1))m^{40}_2-(1\ot \al y)m^{31}_1+(1\ot \si(\be_2))\\
	&\varphantom{=}{}\cdot\Dell_2^{\si\pl_2}(\vphi)m^{22}_2]+\difr z_1[(1\ot \al y)m^{40}_1-(1\ot \si(\be_1))\De_1^{\si\pl_1}(\vphi)m^{31}_3-(1\ot \si(\be_2))\\
	&\varphantom{=}{}\cdot\De_2^{\si\pl_2}(\vphi)m^{31}_4]\\
	&=(1\ot \si(\be_1))[-\Delr_2(\vphi)m^{31}_4-(1\ot y)m^{40}_1+(y\ot 1)m^{40}_2-\De_1^{\si\pl_1}(\vphi)\difr z_1m^{31}_3]\\
	&\varphantom{=}{}+(1\ot \si(\be_2))[\Delr_2(\vphi)m^{31}_3-\De_2^{\si\pl_2}(\vphi)(1\ot y)m^{22}_1+\De_2^{\si\pl_2}(\vphi)(y\ot 1)m^{22}_2\\
	&\varphantom{=}{}-\De_2^{\si\pl_2}(\vphi)\difr z_1m^{31}_4]+(1\ot\al y)[-\De_2(\vphi)m^{22}_1-(y\ot 1)m^{31}_1+\diffd z_1m^{40}_1]\\
	&\overset{\dagger_5}{=}(1\ot \si(\be_1))[\Delr_1(\vphi)m^{31}_3-\De_1^{\si\pl_1}(\vphi)\difr z_1m^{31}_3]+(1\ot \si(\be_2))[\Delr_2(\vphi)m^{31}_3-\De_2^{\si\pl_2}(\vphi)\\
	&\varphantom{=}{}\cdot\difr z_2m^{31}_3]+(1\ot\al y)(1\ot x)m^{31}_3\\
	&\overset{\dagger_6}{=}(1\ot \si(\be_1))(1\ot\si(\vphi_1))m^{31}_3+(1\ot \si(\be_2))(1\ot\si(\vphi_2))m^{31}_3+(1\ot\si(\al \vphi))m^{31}_3\\
	&=m^{31}_3
	\end{align*}
	where $\dagger_5$ follows from \eqref{eq:cocyc-8}, \eqref{eq:cocyc-2}, \eqref{eq:cocyc-3}, and $\dagger_6$ from Lemma \ref{lem:Del-partial} (3). For the last,
	\begin{align*}
	&\varphantom{=}\Delr_1(\vphi)n^{12}_2-(1\ot y)n^{21}_2+(y\ot 1)n^{21}_4+\difr z_2n^{30}_2\\
	&=\Delr_1(\vphi)[(1\ot \si(\be_1))m^{31}_4-(1\ot \si(\be_2))m^{31}_3
	+(1\ot \al y)m^{22}_1]-(1\ot y)[(1\ot \be_2)m^{40}_1\\
	&\varphantom{=}{}-(1\ot \be_1)\De_1^{\pl_1}(\vphi)m^{22}_1]+(y\ot 1)[(1\ot \si(\be_2))m^{40}_2-(1\ot \al y)m^{31}_2-(1\ot \si(\be_1))\\
	&\varphantom{=}{}\cdot\Dell_1^{\si\pl_1}(\vphi)m^{22}_2]+\difr z_2[(1\ot \al y)m^{40}_1-(1\ot \si(\be_1))\De_1^{\si\pl_1}(\vphi)m^{31}_3-(1\ot \si(\be_2))\\
	&\varphantom{=}{}\cdot\De_2^{\si\pl_2}(\vphi)m^{31}_4]\\
	&=(1\ot \si(\be_1))[\Delr_1(\vphi)m^{31}_4-\De_1^{\si\pl_1}(\vphi)(1\ot y)m^{22}_1+\De_1^{\si\pl_1}(\vphi)(y\ot 1)m^{22}_2-\De_1^{\si\pl_1}(\vphi)\\
	&\varphantom{=}{}\cdot\difr z_2m^{31}_3]+(1\ot \si(\be_2))[-\Delr_1(\vphi)m^{31}_3-(1\ot y)m^{40}_1+(y\ot 1)m^{40}_2-\De_2^{\si\pl_2}(\vphi)\\
	&\varphantom{=}{}\cdot\difr z_2m^{31}_4]+(1\ot\al y)[\De_1(\vphi)m^{22}_1-(y\ot 1)m^{31}_2+\diffd z_2m^{40}_1]\\
	&\overset{\dagger_7}{=}(1\ot \si(\be_1))[\Delr_1(\vphi)m^{31}_4-\De_1^{\si\pl_1}(\vphi)\difr z_1m^{31}_4]+(1\ot \si(\be_2))[\Delr_2(\vphi)m^{31}_4-\De_2^{\si\pl_2}(\vphi)\\
	&\varphantom{=}{}\cdot\difr z_2m^{31}_4]+(1\ot\al \vphi)(1\ot x)m^{31}_4\\
	&\overset{\dagger_8}{=}(1\ot \si(\be_1))(1\ot\si(\vphi_1))m^{31}_4+(1\ot \si(\be_2))(1\ot\si(\vphi_2))m^{31}_4+(1\ot\si(\al \vphi))m^{31}_4\\
	&=m^{31}_4
	\end{align*}
	where $\dagger_7$ follows from \eqref{eq:cocyc-2}, \eqref{eq:cocyc-8}, \eqref{eq:cocyc-4}, and $\dagger_8$ from Lemma \ref{lem:Del-partial} (3).
\end{proof}

\begin{lemma}
	Eq.\ \eqref{eq:m-n-3} holds true.
\end{lemma}

\begin{proof}
	We have
	\begin{align*}
	&\varphantom{=}\De_1(\vphi)n^{21}_1+\De_2(\vphi)n^{21}_2+(y\ot 1)n^{30}_1+(1\ot x)n^{30}_2\\
	&=\De_1(\vphi)[(1\ot \be_1)m^{40}_1+(1\ot \be_2)\De_2^{\pl_2}(\vphi)m^{22}_1]+\De_2(\vphi)[(1\ot \be_2)m^{40}_1-(1\ot \be_1)\\
	&\varphantom{=}{}\cdot\De_1^{\pl_1}(\vphi)m^{22}_1]+(y\ot 1)[-(1\ot \be_1)\Dell_1^{\pl_1}(\vphi)m^{31}_1-(1\ot\be_2)\Dell_2^{\pl_2}(\vphi)m^{31}_2]\\
	&\varphantom{=}{}+(1\ot x)[(1\ot \al y)m^{40}_1-(1\ot \si(\be_1))\De_1^{\si\pl_1}(\vphi)m^{31}_3-(1\ot \si(\be_2))\De_2^{\si\pl_2}(\vphi)m^{31}_4]\\
	&=(1\ot \be_1)[\De_1(\vphi)m^{40}_1-\De_1^{\pl_1}(\vphi)\De_2(\vphi)m^{22}_1-\De_1^{\pl_1}(\vphi)(y\ot 1)m^{31}_1-\De_1^{\pl_1}(\vphi)\\
	&\varphantom{=}{}\cdot(1\ot x)m^{31}_3]+(1\ot \be_2)[\De_2^{\pl_2}(\vphi)\De_1(\vphi)m^{22}_1+\De_2(\vphi)m^{40}_1-\De_2^{\pl_2}(\vphi)(y\ot 1)m^{31}_2\\
	&\varphantom{=}{}-\De_2^{\pl_2}(\vphi)(1\ot x)m^{31}_4]+(1\ot\al yx)m^{40}_1\\
	&\overset{\dagger_1}{=}(1\ot \be_1)[\De_1(\vphi)m^{40}_1-\De_1^{\pl_1}(\vphi)\diffd z_1m^{40}_1]+(1\ot \be_2)[\De_2(\vphi)m^{40}_1-\De_2^{\pl_2}(\vphi)\diffd z_2m^{40}_1]\\
	&\varphantom{=}{}+(1\ot\al \vphi)m^{40}_1\\
	&\overset{\dagger_2}{=}(1\ot \be_1)(1\ot \vphi_1)m^{40}_1+(1\ot \be_2)(1\ot \vphi_2)m^{40}_1+(1\ot\al \vphi)m^{40}_1\\
	&=m^{40}_1
	\end{align*}
	where $\dagger_1$ follows from \eqref{eq:cocyc-3}, \eqref{eq:cocyc-4}, and $\dagger_2$ from Lemma \ref{lem:Del-partial} (1). Furthermore,
	\begin{align*}
		&\varphantom{=}\Delb_1(\vphi)n^{21}_3+\Delb_2(\vphi)n^{21}_4+(1\ot y)n^{30}_1+(x\ot 1)n^{30}_2\\
		&=\Delb_1(\vphi)[(1\ot \si(\be_1))m^{40}_2-(1\ot \al y)m^{31}_1+(1\ot \si(\be_2))\Dell_2^{\si\pl_2}(\vphi)m^{22}_2]\\
		&\varphantom{=}{}+\Delb_2(\vphi)[(1\ot \si(\be_2))m^{40}_2-(1\ot \al y)m^{31}_2-(1\ot \si(\be_1))\Dell_1^{\si\pl_1}(\vphi)m^{22}_2]\\
		&\varphantom{=}{}+(1\ot y)[-(1\ot \be_1)\Dell_1^{\pl_1}(\vphi)m^{31}_1-(1\ot\be_2)\Dell_2^{\pl_2}(\vphi)m^{31}_2]\\
		&\varphantom{=}{}+(x\ot 1)[(1\ot \al y)m^{40}_1-(1\ot \si(\be_1))\De_1^{\si\pl_1}(\vphi)m^{31}_3-(1\ot \si(\be_2))\De_2^{\si\pl_2}(\vphi)m^{31}_4]\\
		&=(1\ot \si(\be_1))[\Delb_1(\vphi)m^{40}_2-\Dell_1^{\si\pl_1}(\vphi)\Delb_2(\vphi)m^{22}_2-\Dell_1^{\si\pl_1}(\vphi)(1\ot y)m^{31}_1\\
		&\varphantom{=}{}-\Dell_1^{\si\pl_1}(\vphi)(x\ot 1)m^{31}_3]+(1\ot \si(\be_2))[\Dell_2^{\si\pl_2}(\vphi)\Delb_1(\vphi)m^{22}_2+\Delb_2(\vphi)m^{40}_2\\
		&\varphantom{=}{}-\Dell_2^{\si\pl_2}(\vphi)(1\ot y)m^{31}_2-\Dell_2^{\si\pl_2}(\vphi)(x\ot 1)m^{31}_4]+(1\ot\al y)[-\Dell_1(\vphi)m^{31}_1\\
		&\varphantom{=}{}-\Dell_2(\vphi)m^{31}_2+(x\ot 1)m^{40}_1]\\
		&\overset{\dagger_3}{=}(1\ot \si(\be_1))[\Delb_1(\vphi)m^{40}_2-\Dell_1^{\si\pl_1}(\vphi)\difb z_1m^{40}_2]+(1\ot \si(\be_2))[\Delb_2(\vphi)m^{40}_2\\
		&\varphantom{=}{}-\Dell_2^{\si\pl_2}(\vphi)\difb z_2m^{40}_2]+(1\ot\si(\al \vphi))m^{40}_2\\
		&\overset{\dagger_4}{=}(1\ot \si(\be_1))(1\ot \si(\vphi_1))m^{40}_2+(1\ot \si(\be_2))(1\ot \si(\vphi_2))m^{40}_2+(1\ot\si(\al \vphi))m^{40}_2\\
		&=m^{40}_2
	\end{align*}
	where $\dagger_3$ follows from \eqref{eq:cocyc-5}, \eqref{eq:cocyc-6}, \eqref{eq:cocyc-7}, and $\dagger_4$ from Lemma \ref{lem:Del-partial} (4).
\end{proof}

\subsection{Proof of necessity}

Now suppose $(\vphi,\vphi_1,\vphi_2)\neq B$. Let us first consider the situation that $\kk$ is algebraically closed. The arguments presented below are mainly inspired by \cite{Bavula:gldim-polyn}.

\begin{lemma}\cite{Bavula:gldim-polyn}\label{lem:infty-gld-1}
	If $\vphi=0$, then then $W_{(2)}$ has infinity global dimension. 
\end{lemma}

\begin{lemma}\label{lem:infty-gld-2}
	Suppose that $\kk$ is algebraically closed and $(\vphi,\vphi_1,\vphi_2)\neq B$. If $\vphi\neq 0$ then $W_{(2)}$ has infinity global dimension. 
\end{lemma}


\begin{proof}
	Recall from Proposition \ref{prop:free-resol-GWA} that $\Tot \mc{P}_{\sbu\sbu}$ is alternate. For any right $W_{(2)}$-module $M$ and any left $W_{(2)}$-module $N$, we have $\Tor^{W_{(2)}}_4(M,N)=\Tor^{W_{(2)}}_6(M,N)=\Tor^{W_{(2)}}_8(M,N)=\cdots$. So, it suffices to show that there exist $M$ and $N$ such that $\Tor^{W_{(2)}}_4(M,N)\neq 0$. 
	
Since $(\vphi,\vphi_1,\vphi_2)$ is proper and $\kk$ is algebraically closed, there exists a maximal ideal $\mf{m}=(z_1-\lam_1,z_2-\lam_2)$ in $B$ containing $(\vphi,\vphi_1,\vphi_2)$. Let $\varepsilon_{\mf{m}}\colon B\to B/\mf{m}$ be the projection. We claim that $(\varepsilon_{\mf{m}}\ot\varepsilon_{\mf{m}})(\De_1(\vphi))=(\varepsilon_{\mf{m}}\ot\varepsilon_{\mf{m}})(\De_2(\vphi))=0$. In fact, we write $\De_1(\vphi)=\sum b'\ot b''$ using Sweedler's notation, then
\begin{align*}
(\varepsilon_{\mf{m}}\ot\varepsilon_{\mf{m}})(\De_1(\vphi))&=\sum\varepsilon_{\mf{m}}(b')\ot\varepsilon_{\mf{m}}(b'')=\sum\varepsilon_{\mf{m}}(b')\varepsilon_{\mf{m}}(b'')\ot\bar{1}\\
&=\varepsilon_{\mf{m}}\biggl(\sum b'b''\biggr)\ot\bar{1}=\varepsilon_{\mf{m}}(\vphi_1)\ot\bar{1}=0
\end{align*}
and similarly for $\De_2(\vphi)$.

Denote by $I_r$ and $I_l$ respectively the right ideal in $W_{(2)}$ generated by $y$, $z_1-\lam_1$, $z_2-\lam_2$, and the left ideal in $W_{(2)}$ generated by $x$, $z_1-\lam_1$, $z_2-\lam_2$. By analyzing the basis elements of $W_{(2)}$, we know that as $\kk$-modules,
\begin{gather*}
I_r=\bigoplus_{i\geq 1}By^i\oplus\bigoplus_{j\geq 0}\mf{m}x^j, \\
I_l=\bigoplus_{i\geq 0}\mf{m}y^i\oplus\bigoplus_{j\geq 1}Bx^j,
\end{gather*}
which are both proper. Hence the right $W_{(2)}$-module $M=W_{(2)}/I_r$ and the left $W_{(2)}$-module $N=W_{(2)}/I_l$ are nonzero. Moreover, by abuse of notations, we also refer to $\e_{\mf m}$ the map $B\to M$ or $B\to N$ by identifying $B/\mf m$ with the summand of $M$ or $N$. 

Let us consider $\Tor^{W_{(2)}}_4(M, N)$, which is the fourth homology group of $M\ot_{W_{(2)}}(\Tot\mc P_{\sbu\sbu})\ot_{W_{(2)}}N$. The following is a part of $M\ot_{W_{(2)}}\mc P_{\sbu\sbu}\ot_{W_{(2)}}N$:
\[
\xymatrix{
	(M\ot N)^4 & & 	(M\ot N)^4 \ar[d]^-{d^v_{40}} \\
	& (M\ot N)^2 & \boxed{(M\ot N)^2} \ar[l]^-{d^h_{30}}\ar[llu]_-{t_{21}} & (M\ot N)^2 \ar[l]^-{d^h_{40}}
	}
\]
with $M\ot_{W_{(2)}}\mc P_{40}\ot_{W_{(2)}}N$ displayed in box. Consider the element $(\bar{1}\ot\bar{1}, 0)$ in $M\ot_{W_{(2)}}\mc P_{40}\ot_{W_{(2)}}N$. We have
\begin{align*}
t_{21}(\bar{1}\ot\bar{1}, 0)&=((\varepsilon_{\mf{m}}\ot\varepsilon_{\mf{m}})(\De_1(\vphi)),(\varepsilon_{\mf{m}}\ot\varepsilon_{\mf{m}})(\De_2(\vphi)),0,0)=0,\\
d^h_{30}(\bar{1}\ot\bar{1}, 0)&=(\bar{y}\ot \bar{1}, \bar{1}\ot \bar{x})=0,
\end{align*}
and thus $(\bar{1}\ot\bar{1}, 0)$ gives rise to a $4$-cycle in $M\ot_{W_{(2)}}(\Tot\mc P_{\sbu\sbu})\ot_{W_{(2)}}N$. Were $(\bar{1}\ot\bar{1}, 0)$ a boundary, it would be of the form $d^v_{40}(a_1,a_2,a_3,a_4)+d^h_{40}(b_1,b_2)$. It follows that
\begin{align*}
\bar{1}\ot\bar{1}&=\sum (a_1'\triangleleft z_1\ot a_1''-a_1'\ot z_1\triangleright a_1'')+\sum (a_2'\triangleleft z_2\ot a_2''-a_2'\ot z_2\triangleright a_2'')\\
&\varphantom{=}{}-\sum b_1'\triangleleft x\ot b_1''+\sum b_2'\ot y\triangleright b_2''.
\end{align*}
Since $W_{(2)}$ is a $\inn$-graded algebra by setting $|x|=1$ and $|y|=-1$, $M$, $N$ are graded $W_{(2)}$-modules with $M=M_{\geq 0}=\oplus_{j\geq 0}\kk x^j$ and $N=N_{\leq 0}=\oplus_{i\geq 0}\kk y^i$. By taking degree into account, we may assume $a_1$, $a_2\in M_0\ot N_0$, $b_1\in M_{-1}\ot N_0$, and $b_2\in M_0\ot N_1$. This forces $b_1=b_2=0$, and $a_1'$, $a_1''$, $a_2'$, $a_2''\in B/\mf{m}$. Thus the above equality is simplified into
\[
\bar{1}\ot\bar{1}=\sum (a_1'\lam_1\ot a_1''-a_1'\ot \lam_1a_1'')+\sum (a_2'\lam_2\ot a_2''-a_2'\ot \lam_2a_2'')=0,
\]
a contradiction.

Thus we catch a nontrivial 4-class and hence $\Tor^{W_{(2)}}_4(M,N)\neq 0$, as desired. 
\end{proof}

\begin{proposition}
	If $(\vphi,\vphi_1,\vphi_2)\neq B$, then $W_{(2)}$ has infinite global dimension. Accordingly, $W_{(2)}$ is not homologically smooth.
\end{proposition}

\begin{proof}
	By Lemmas \ref{lem:infty-gld-1}, \ref{lem:infty-gld-2}, the proposition is true when $\kk$ is algebraically closed.
	
	For the general case, let $\mbb K$ be the algebraic closure of $\kk$ and $W_{(2)}^{\mbb K}=\mbb K\ot W_{(2)}$. Note that $W_{(2)}^{\mbb K}$ is a GWA over $B^{\mbb K}=\mbb K[z_1,z_2]$ and the ideal $(\vphi,\vphi_1,\vphi_2)$ in $B^{\mbb{K}}$ is proper. According to our argument above, the $\mbb K$-algebra $W_{(2)}^{\mbb K}$ has infinite global dimension. Since the $\Tor$ groups respect localization, i.e., $\mbb{K}\otimes\Tor_*^{W_{(2)}}(M, N)\cong \Tor_*^{W_{(2)}^{\mbb K}}(M^{\mbb K}, N^{\mbb K})$ for all $M$, $N$, there exist $M_n$ and $N_n$ such that $\Tor_n^{W_{(2)}}(M_n, N_n)\neq 0$ for all $n\in\nan$. Therefore we finish the proof in the general case.
\end{proof}

\section{Nakayama automorphisms}\label{sec:Nak-auto}

In \cite{Van-den-Bergh:VdB-duality}, Van den Bergh proved the existence of a duality between Hochschild homology and cohomology for a class of Gorenstein algebras $A$ under the homological smoothness condition. Namely, there exists an invertible $A$-bimodule $U$ and a positive integer $d$ such that
\[
H^i(A,M)\cong H_{d-i}(A,U\ot_A M)
\]
naturally holds for all $A^e$-modules $M$ and all integers $i$.  In particular, if the invertible bimodule $U$ is of the form $A^\nu$ for some $\nu\in\Aut(A)$, then the duality becomes
\begin{equation}\label{eq:twisted-Poincare}
H^i(A,M)\cong H_{d-i}(A, M^\nu).
\end{equation}
This is usually called twisted Poincar\'e duality in the literature. Recall that an algebra $A$ enjoys twisted Poincar\'e duality if $A$ is a twisted $d$-Calabi-Yau algebra, namely, $A$ is homologically smooth and the condition
\begin{equation}\label{eq:NAK}
\Ext_{A^e}^i(A, A^e)\cong\begin{cases}
0, & i\neq d, \\ A^\nu, & i=d
\end{cases}
\end{equation}
is fulfilled for some $d\in\nan$ and $\nu\in\Aut(A)$. The number $d$ is called the Hochschild cohomology dimension of $A$ and respectively $\nu$ the Nakayama automorphism of $A$, which coincide the $d$, $\nu$ in \eqref{eq:twisted-Poincare}. 

Many classes of algebras arising from noncommutative algebraic geometry or quantum group are twisted Calabi-Yau. We refer to \cite{Bocklandt-Schedler-Wemyss:superpotential}, \cite{Brown-Zhang:AS-Gorenstein-Hopf}, \cite{Chan-Walton-Zhang:NAK}, \cite{Ginzburg:CY-alg}, \cite{Krahmer:qh-space}, \cite{L-Wu:rdc-qhs}, \cite{VdB:CY-alg}, \cite{Yekutieli:rigid-dualizing-complex-universal-enveloping-algebra} and the references therein for more information and in particular plenty of examples.

\begin{remark}
	\begin{enumerate}
		\item In \eqref{eq:NAK}, the second variable $A^e$ in the Ext group has left and right $A^e$-module structures. The left is used for computing Ext, and the right is survival, inducing the $A$-bimodule structure on the Ext group.
		\item An algebra $A$ is said to have a Nakayama automorphism $\nu$ if \eqref{eq:NAK} is satisfied, even if $A$ is not homologically smooth. The  automorphism $\nu$ is unique up to inner automorphism. In \cite{Lv-Mao-Zhang:NAK}, $\nu$ is proven to be central in $\Aut(A)/\mathrm{Inn}(A)$.
		\item When $\nu$ is inner, the algebra $A$ is Calabi-Yau in the sense of Ginzburg \cite{Ginzburg:CY-alg}.
	\end{enumerate}
\end{remark}

Let us focus on Nakayama automorphisms of GWAs. Recall that a GWA $W_{(1)}$  is proven to have a Nakayama automorphism (with $d=2$) whose explicit expression is determined in \cite{L:gwa-def}. By adapting the proof in loc.\ cit.\ one can conclude that $W_{(2)}$ also has a Nakayama automorphism (with $d=3$). However, one does not know the expressions of $\si(z_1)$, $\si(z_2)$ for an arbitrary $\si\in\Aut(B)$, as is mentioned in Remark \ref{rmk:two-smoothness}. Due to the indeterminacy of $\si$, we are not able to capture the Nakayama automorphism $\nu$ as we did in \cite{L:gwa-def}. Instead, we will deduce the expression of $\nu$ by spectral sequence argument.

It is illustrated in \cite{L:gwa-def} that the filtration by column of a homotopy double complex gives rise to a spectral sequence. If the homotopy double complex sits in the first quadrant, then the spectral sequence converges to the (co)homology of the associated total complex. So let us apply it to the homotopy double complex $\mc{Q}_{W^e}^{\sbu\sbu}$. Denote by $E^{pq}_{\sbu}$ the induced spectral sequence.

Since $W_{(2)}\cong B^{(\nan)}$ as a left $B$-module, we have $W_{(2)}^e\cong (B^e)^{(\nan\times\nan)}\cong (B^e)^{(\nan)}$ as a left $B^e$-module. Observe that $B$ is 2-Calabi-Yau and so
\[
E_1^{0q}=H^q(\mc{Q}_{W_{(2)}^e}^{0\sbu})=\Ext^q_{W_{(2)}^e}(W_{(2)}^e\ot_{B^e}B, W_{(2)}^e)\cong\Ext^q_{B^e}(B, B^e)^{(\nan)}
\]
is zero unless $q=2$. By a similar manner, we conclude that for all $p\geq 1$, $E^{pq}_1$ is nonzero only if $q= 2$.

Notice that $E^{02}_1$ is a quotient of $E^{02}_0=W_{(2)}^e$. It is easy to prove that
\[
E^{02}_1=\Bigl\{\sum [w'\ot w'']\Bigm| \text{$w'\in W_{(2)}$ and $w''$ is a power of $x$ or $y$}\Bigr\}
\]
as a $\kk$-module, and similarly, we have $E^{p2}_1=E^{02}_1\oplus E^{02}_1$ for all $p\geq 1$. Hence, the differentials $E^{02}_1\xrightarrow{d^0}E^{12}_1\xrightarrow{d^1}E^{2}_1$ are respectively given by
\[
d^0=\begin{pmatrix}
-x\ot 1+J\ot x \\ 1\ot y-Jy\ot 1
\end{pmatrix}
\qquad\text{and}\qquad
d^1=\begin{pmatrix}
y\ot 1 & 1\ot x \\ 1\ot y & x\ot 1
\end{pmatrix}.
\]
Here we remind the reader that $J$ is the Jacobian determinant of $\si$ and that $d^0$ is induced by $d^h_{02}$ given in \S\ref{subsec:constr-homo-double}, and their matrix representations are in fact equal up to $\im d_v^{02}$.

We begin to compute $E_2^{12}$. Suppose $\mbf{x}:=(\mbf{x}_1,\mbf{x}_2)^T\in\ker d^2$ with
\begin{align*}
\mbf{x}_1&=\sum_{i,j\geq 1}[a^1_{ij}x^i\ot x^j]+\sum_{i\geq 1,j\geq 0}[a^2_{ij}x^i\ot y^j]+\sum_{i\geq 0,j\geq 1}[a^3_{ij}y^i\ot x^j]+\sum_{i,j\geq 0}[a^4_{ij}y^i\ot y^j], \\ 
\mbf{x}_2&=\sum_{i,j\geq 0}[b^1_{ij}x^i\ot x^j]+\sum_{i\geq 0,j\geq 1}[b^2_{ij}x^i\ot y^j]+\sum_{i\geq 1,j\geq 0}[b^3_{ij}y^i\ot x^j]+\sum_{i,j\geq 1}[b^4_{ij}y^i\ot y^j].
\end{align*}
Then by a direct computation, we have
\begin{gather*}
b^1_{ij}=-\vphi\si^{-1}(a^1_{i+1,j+1}),\; b^2_{ij}=-\si^{-1}(a^2_{i+1,j-1}),\\
b^3_{ij}=-\si^{-1}(a^3_{i-1,j+1}),\; a^4_{ij}=-\si(\vphi)\si(b^4_{i+1,j+1})
\end{gather*}
and hence
\begin{align*}
\mbf{x}&=
\sum_{i\geq 1,j\geq 1}\begin{pmatrix}
[a^1_{ij}x^i\ot x^j] \\ [-\vphi\si^{-1}(a^1_{ij})x^{i-1}\ot x^{j-1}]
\end{pmatrix}+\sum_{i\geq 1,j\geq 0}\begin{pmatrix}[a^2_{ij}x^i\ot y^j] \\ [-\si^{-1}(a^2_{ij})x^{i-1}\ot y^{j+1}] \end{pmatrix}\\
&\varphantom{=}{}+\sum_{i\geq 0,j\geq 1}\begin{pmatrix}[a^3_{ij}y^i\ot x^j]\\ [-\si^{-1}(a^3_{ij})y^{i+1}\ot x^{j-1}]\end{pmatrix}+\sum_{i\geq 0,j\geq 0}\begin{pmatrix}[-\si(\vphi)\si(\tilde{b}^4_{ij})y^i\ot y^j] \\ [\tilde{b}^4_{ij}y^{i+1}\ot y^{j+1}] \end{pmatrix}
\end{align*}
where $\tilde{b}^4_{ij}=b^4_{i+1,j+1}$. On the other hand, it is routine to check
\begin{align*}
\begin{pmatrix}
[a^1_{ij}x^i\ot x^j] \\ [-\vphi\si^{-1}(a^1_{ij})x^{i-1}\ot x^{j-1}]
\end{pmatrix}+\im d^0&=
\begin{pmatrix}
[J^{-1}\si(a^1_{ij})x^{i+1}\ot x^{j-1}] \\ [-J^{-1}\vphi a^1_{ij}x^{i}\ot x^{j-2}]
\end{pmatrix}+\im d^0,\\
\begin{pmatrix}[a^2_{ij}x^i\ot y^j] \\ [-\si^{-1}(a^2_{ij})x^{i-1}\ot y^{j+1}] \end{pmatrix}+\im d^0&=
\begin{pmatrix}[J\si(\vphi)\si^{-1}(a^2_{ij})x^{i-1}\ot y^{j-1}] \\ [-J\vphi\si^{-2}(a^2_{ij})x^{i-2}\ot y^{j}] \end{pmatrix}+\im d^0,\\
\begin{pmatrix}[a^3_{ij}y^i\ot x^j]\\ [-\si^{-1}(a^3_{ij})y^{i+1}\ot x^{j-1}]\end{pmatrix}+\im d^0&=\begin{pmatrix}[J^{-1}\si(\vphi)\si(a^3_{ij})y^{i-1}\ot x^{j-1}]\\ [-J^{-1}\vphi a^3_{ij}y^{i}\ot x^{j-2}]\end{pmatrix}+\im d^0,\\
\begin{pmatrix}[-\si(\vphi)\si(\tilde{b}^4_{ij})y^i\ot y^j] \\ [\tilde{b}^4_{ij}y^{i+1}\ot y^{j+1}] \end{pmatrix}+\im d^0&=
\begin{pmatrix}[-J\si(\vphi)\tilde{b}^4_{ij}y^{i+1}\ot y^{j-1}] \\ [J\si^{-1}(\tilde{b}^4_{ij})y^{i+2}\ot y^{j}] \end{pmatrix}+\im d^0.
\end{align*}
Henceforth, $\mbf{x}+\im d^0$ can be uniquely expressed as $\mbf{y}+\im d^0$ with
\begin{equation}\label{eq:Phi-y}
\mbf{y}=
\sum_{i\geq 1}\begin{pmatrix}[c^1_{i}x^i\ot 1] \\ [-\si^{-1}(c^1_{i})x^{i-1}\ot y] \end{pmatrix}+\sum_{i\geq 1}\begin{pmatrix}[c^2_{i}y^i\ot x] \\ [-\si^{-1}(c^2_{i})y^{i+1}\ot 1] \end{pmatrix}.
\end{equation}

Since $\Ext_{W_{(2)}^e}^3(W_{(2)},W_{(2)}^e)\cong E_2^{12}$, we have to study the $W_{(2)}$-bimodule structure on $E_2^{12}$ which is induced by the right regular module structure on $W_{(2)}^e$.\footnote{This is also called the inner bimodule structure on $W_{(2)}\ot W_{(2)}$.} Let us write an isomorphism $\Phi\colon E_2^{12}\to W_{(2)}$ of $\kk$-module explicitly as follows: for any class $\mbf{\bar{y}}:=\mbf{y}+\im d^0\in E_2^{12}$ as above, define
\[
\Phi(\mbf{\bar{y}})=\sum_{i\geq 1}J^{i}\si^{-1}(c^1_{i})x^{i-1}+\sum_{i\geq 1}J^{-i}c^2_{i}y^{i}.
\]

\begin{lemma}
Let $J$ be the Jacobian determinant of $\si$, and $\nu\colon W_{(2)}\to W_{(2)}$ be the automorphism defined by
\begin{equation}\label{eq:NAK-auto-formula}
\nu(x)=Jx,\quad \nu(y)=J^{-1}y,\quad \nu(z_1)=z_1,\quad \nu(z_2)=z_2.
\end{equation}
Then $\Phi\colon E^{12}_2\to W_{(2)}^\nu$ is an isomorphism of $W_{(2)}$-bimodules.
\end{lemma}

\begin{proof}
	Let us prove that $\Phi$ is an isomorphism of right $W_{(2)}$-modules. The left modules case is similar.
	
	Any class in $E_2^{12}$ can be represented by some $\mbf{y}$ as in \eqref{eq:Phi-y}. So we need to verify $\Phi(\mbf{\bar{y}}\triangleleft w)=\Phi(\mbf{\bar{y}})\triangleleft w$ for $w\in\{z_1,z_2,x,y\}$.
	
	Let $z=z_1$ or $z_2$. We have
	\begin{align*}
	\Phi(\mbf{\bar{y}}\triangleleft z)&=\Phi\!\left(\sum_{i\geq 1}\begin{pmatrix}[c^1_{i}x^iz\ot 1] \\ [-\si^{-1}(c^1_{i})x^{i-1}z\ot y] \end{pmatrix}+\sum_{i\geq 1}\begin{pmatrix}[c^2_{i}y^iz\ot x] \\ [-\si^{-1}(c^2_{i})y^{i+1}z\ot 1] \end{pmatrix}\!\right)\\
	&=\Phi\!\left(\sum_{i\geq 1}\begin{pmatrix}[c^1_{i}\si^i(z)x^i\ot 1] \\ [-\si^{-1}(c^1_{i}\si^i(z))x^{i-1}\ot y] \end{pmatrix}\!+\!\sum_{i\geq 1}\begin{pmatrix}[c^2_{i}\si^{-i}(z)y^i\ot x] \\ [-\si^{-1}(c^2_{i}\si^{-i}(z))y^{i+1}\ot 1] \end{pmatrix}\!\right)\\
	&=\sum_{i\geq 1}J^{i}\si^{-1}(c^1_{i}\si^i(z))x^{i-1}+\sum_{i\geq 1}J^{-i}c^2_{i}\si^{-i}(z)y^{i}\\
	&=\sum_{i\geq 1}J^{i}\si^{-1}(c^1_{i})x^{i-1}z+\sum_{i\geq 1}J^{-i}c^2_{i}y^{i}z\\
	&=\Phi(\mbf{\bar{y}})\triangleleft z.
	\end{align*}
	Also, we have
	\begin{align*}
	\Phi(\mbf{\bar{y}}\triangleleft x)&=\Phi\left(\sum_{i\geq 1}\begin{pmatrix}[c^1_{i}x^{i+1}\ot 1] \\ [-\si^{-1}(c^1_{i})x^{i}\ot y] \end{pmatrix}+\sum_{i\geq 1}\begin{pmatrix}[c^2_{i}y^ix\ot x] \\ [-\si^{-1}(c^2_{i})y^{i+1}x\ot 1] \end{pmatrix}\right)\\
	&=\Phi\left(\sum_{i\geq 2}\begin{pmatrix}[c^1_{i-1}x^{i}\ot 1] \\ [-\si^{-1}(c^1_{i-1})x^{i-1}\ot y] \end{pmatrix}+\sum_{i\geq 1}\begin{pmatrix}[c^2_{i}\si^{1-i}(\vphi)y^{i-1}\ot x] \\ [-\si^{-1}(c^2_{i}\si^{1-i}(\vphi))y^{i}\ot 1] \end{pmatrix}\right)\\
	&=\Phi\left(\sum_{i\geq 2}\begin{pmatrix}[c^1_{i-1}x^{i}\ot 1] \\ [-\si^{-1}(c^1_{i-1})x^{i-1}\ot y] \end{pmatrix}+\begin{pmatrix}[c^2_{1}\vphi \ot x] \\ [-\si^{-1}(c^2_{1}\vphi)y\ot 1] \end{pmatrix}\right.\\
	&\varphantom{=}{}+\left.\sum_{i\geq 1}\begin{pmatrix}[c^2_{i+1}\si^{-i}(\vphi)y^{i}\ot x] \\ [-\si^{-1}(c^2_{i+1}\si^{-i}(\vphi))y^{i+1}\ot 1] \end{pmatrix}\right)\\
	&=\Phi\left(\sum_{i\geq 2}\begin{pmatrix}[c^1_{i-1}x^{i}\ot 1] \\ [-\si^{-1}(c^1_{i-1})x^{i-1}\ot y] \end{pmatrix}+\begin{pmatrix}[J^{-1}\si(c^2_{1}\vphi)x \ot 1] \\ [-J^{-1}c^2_{1}\vphi\ot y] \end{pmatrix}\right.\\
	&\varphantom{=}{}+\left.\sum_{i\geq 1}\begin{pmatrix}[c^2_{i+1}\si^{-i}(\vphi)y^{i}\ot x] \\ [-\si^{-1}(c^2_{i+1}\si^{-i}(\vphi))y^{i+1}\ot 1] \end{pmatrix}\right)\\
	&=\sum_{i\geq 2}J^{i}\si^{-1}(c^1_{i-1})x^{i-1}+J\si^{-1}(J^{-1}\si(c^2_{1}\vphi))+\sum_{i\geq 1}J^{-i}c^2_{i+1}\si^{-i}(\vphi)y^{i}\\
	&=\sum_{i\geq 1}J^{i+1}\si^{-1}(c^1_{i})x^{i}+c^2_{1}\vphi+\sum_{i\geq 1}J^{-i}c^2_{i+1}y^{i}\vphi\\
	&=\sum_{i\geq 1}J^{i}\si^{-1}(c^1_{i})x^{i-1}Jx+\sum_{i\geq 1}J^{-i}c^2_{i}y^{i}Jx\\
	&=\Phi(\mbf{\bar{y}})\triangleleft  x,
	\end{align*}
	and
	\begin{align*}
	\Phi(\mbf{\bar{y}}\triangleleft y)&=\Phi\left(\sum_{i\geq 1}\begin{pmatrix}[c^1_{i}x^{i}y\ot 1] \\ [-\si^{-1}(c^1_{i})x^{i-1}y\ot y] \end{pmatrix}+\sum_{i\geq 1}\begin{pmatrix}[c^2_{i}y^{i+1}\ot x] \\ [-\si^{-1}(c^2_{i})y^{i+2}\ot 1] \end{pmatrix}\right)\\
	&=\Phi\left(\sum_{i\geq 0}\begin{pmatrix}[c^1_{i+1}x^{i+1}y\ot 1] \\ [-\si^{-1}(c^1_{i+1})x^{i}y\ot y] \end{pmatrix}+\sum_{i\geq 2}\begin{pmatrix}[c^2_{i-1}y^{i}\ot x] \\ [-\si^{-1}(c^2_{i-1})y^{i+1}\ot 1] \end{pmatrix}\right)\\
	&=\Phi\left(\sum_{i\geq 1}\begin{pmatrix}[c^1_{i+1}x^{i+1}y\ot 1] \\ [-\si^{-1}(c^1_{i+1})x^{i}y\ot y] \end{pmatrix}+\begin{pmatrix}[c^1_{1}xy\ot 1] \\ [-\si^{-1}(c^1_{1})y\ot y] \end{pmatrix}\right.\\
	&\varphantom{=}{}+\left.\sum_{i\geq 2}\begin{pmatrix}[c^2_{i-1}y^{i}\ot x] \\ [-\si^{-1}(c^2_{i-1})y^{i+1}\ot 1] \end{pmatrix}\right)\\
	&=\Phi\left(\sum_{i\geq 1}\begin{pmatrix}[c^1_{i+1}\si^{i+1}(\vphi)x^{i}\ot 1] \\ [-\si^{-1}(c^1_{i+1}\si^{i+1}(\vphi))x^{i-1}\ot y] \end{pmatrix}+\begin{pmatrix}[J\si^{-1}(c^1_{1})y\ot x] \\ [-Jc^1_{1}y^2\ot 1] \end{pmatrix}\right.\\
	&\varphantom{=}{}+\left.\sum_{i\geq 2}\begin{pmatrix}[c^2_{i-1}y^{i}\ot x] \\ [-\si^{-1}(c^2_{i-1})y^{i+1}\ot 1] \end{pmatrix}\right)\\
	&=\sum_{i\geq 1}J^{i}\si^{-1}(c^1_{i+1}\si^{i+1}(\vphi))x^{i-1}+J^{-1}J\si^{-1}(c^1_{1})y+\sum_{i\geq 2}J^{-i}c^2_{i-1}y^{i}\\
	&=\sum_{i\geq 1}J^{i}\si^{-1}(c^1_{i+1})x^{i-1}\si(\vphi)+\si^{-1}(c^1_{1})y+\sum_{i\geq 2}J^{-i}c^2_{i-1}y^{i}\\
	&=\sum_{i\geq 0}J^{i}\si^{-1}(c^1_{i+1})x^{i}y+\sum_{i\geq 1}J^{-i-1}c^2_{i}y^{i+1}\\
	&=\sum_{i\geq 1}J^{i}\si^{-1}(c^1_{i})x^{i-1}J^{-1}y+\sum_{i\geq 1}J^{-i}c^2_{i}y^{i}J^{-1}y\\
	&=\Phi(\mbf{\bar{y}})\triangleleft  y. \qedhere
	\end{align*}
\end{proof}

Thus we obtain

\begin{theorem}\label{thm:3-tCY}
	Any GWA $W_{(2)}$ has a Nakayama automorphism $\nu$  given by \eqref{eq:NAK-auto-formula}, namely,
	\[
	\Ext_{W_{(2)}^e}^i(W_{(2)}, W_{(2)}^e)\cong\begin{cases}
		0, & i\neq 3, \\ W_{(2)}^\nu, & i=3.
	\end{cases}
	\]
	In particular, if $W_{(2)}$ is  homologically smooth, then $W_{(2)}$ is twisted $3$-Calabi-Yau.
\end{theorem}

\begin{remark}
Since $W_{(2)}$ is noetherian and has (left and right) injective dimension $3$, an equivalent statement of Theorem \ref{thm:3-tCY} is: the rigid dualizing complex over $W_{(2)}$ is ${}^\nu W_{(2)}[3]$.
\end{remark}

\begin{corollary}\label{cor:3-CY}
	A GWA $W_{(2)}=B(\si,\vphi)$ is $3$-Calabi-Yau if and only if $(\vphi,\vphi_1,\vphi_2)=B$ and  $J=1$.
\end{corollary}

Since Hochschild established the cohomology theory for associative algebras $A$ \cite{Hochschild:cohomo-asso-alg}, the theory has been studied by many mathematicians. Amongst the developments, a structure on $HH^{\sbu}(A):=\oplus_{n\in\nan}HH^n(A)$ which is a differential graded version of Poisson algebra was found by Gerstenhaber for all algebras $A$. The structure is nowadays called Gerstenhaber algebra. Batalin-Vilkovisky algebras are a subclass of Gerstenhaber algebras arising from theoretical physics. A remarkable relationship between Batalin-Vilkovisky structure and Hochschild cohomology was illustrated by Ginzburg \cite{Ginzburg:CY-alg}, saying that  $HH^{\sbu}(A)$ is a Batalin-Vilkovisky algebra for all Calali-Yau algebras $A$. Later on, this result was generalized for some twisted Calabi-Yau algebras \cite{Kowalzig-Krahmer:BV-twisted-CY}, that is,

\begin{theorem}\cite[Thm.\ 1.7]{Kowalzig-Krahmer:BV-twisted-CY}
	If A is a twisted Calabi-Yau algebra with semisimple (namely, diagonalizable) Nakayama automorphism, then the Hochschild cohomology $HH^{\sbu}(A)$ of $A$ is a Batalin-Vilkovisky algebra.
\end{theorem}

According to the expression $\eqref{eq:NAK-auto-formula}$, the Nakayama automorphism $\nu$ of $W$ is semisimple. Hence we have

\begin{corollary}\label{cor:BV}
	The Hochschild cohomology $HH^{\sbu}(W_{(2)})$ of $W_{(2)}$ is a Batalin-Vilkovisky algebra if $W_{(2)}$ is homologically smooth.
\end{corollary}

\begin{remark}
	Note that the defining automorphism $\si\in\Aut(\kk[z])$ of $W_{(1)}$ is necessarily determined by $\si(z)=\lam z+\eta$ for some $\lam\in\kk^\times$, $\eta\in \kk$. One of the results in the author's previous paper \cite{L:gwa-def} is that the Nakayama automorphism of $W_{(1)}$ is given by
	\[
	x\mapsto \lam x,\quad y\mapsto \lam^{-1}y, \quad z\mapsto z.
	\]
	Obviously we have $J=\lam$ in this case. It seems reasonable to conjecture that the analogy exists for all GWA $W_{(n)}$ if $n$ is any positive integer. This will be our future work.
\end{remark}

\section{Examples}\label{sec:app}

In this section, we apply Theorem \ref{thm:main-theorem} to concrete algebras, judging them smooth or not. Most results are known, obtained by other people in different manners.

\subsection{Quantum groups $\mc{O}_q(SL_2)$ and $U(\mf{sl}_2)$}

The definitions of these well-known algebras can be found, for example, in \cite{Kassel:quan-gp} and it is well-known \cite{Bavula:GWA-tensor-product} that they are GWA as discussed in this paper. Thus Theorems \ref{thm:main-theorem} and \ref{thm:3-tCY} show that they are homologically smooth and determine their Nakayama automorphisms. These theorems therefore reproduce \cite[\S6]{Brown-Zhang:AS-Gorenstein-Hopf} where the detailed formulas can be found.

\subsection{Noetherian down-up algebras}

Motivated by the study of posets, Benkart and Roby \cite{Benkart-Roby:du-alg} introduced the notion of a down-up algebra $A(\al,\be,\ga)$. Down-up algebras have been intensively studied in for example \cite{Benkart-Witherspoon:hopf-du-alg}, \cite{Solotar:hoch-(co)homo-down-up}, \cite{Kirkman-Musson-Passman:noeth-du-alg}, \cite{Kulkarni:du-alg-repre} among many other articles. It is shown in \cite{Kirkman-Musson-Passman:noeth-du-alg} that $A(\al,\be,\ga)$ is right (or left) noetherian if and only if $\be \neq 0$. Also in \cite{Kirkman-Musson-Passman:noeth-du-alg}, a noetherian down-up algebra $A(\al,\be,\ga)$ is a GWA. Thus Theorems \ref{thm:main-theorem} and \ref{thm:3-tCY} show that $A(\al,\be,\ga)$ is homologically smooth and determine the Nakayama automorphism, reproducing the formula shown in \cite{Lv-Mao-Zhang:NAK}, \cite{Shen-Lu:Nak-PBW}.

\subsection{A quotient algebra of $M(p,q)$}\label{subsec:Mpq}

In \cite{Chen:example} a noncommutative and noncocommutative bialgebra $M(p,q)$ for two parameters $p$ and $q$ is constructed. The algebra is generated by four elements $a$, $b$, $c$, $d$ satisfying some relations similar to those of the quantum matrix algebra $M_q(2)$. Concretely, these relations are:
\begin{gather*}
ba=qab,\quad dc=qcd,\quad ca=qac,\\ db=qbd,\quad bc=cb,\quad da-qad=p(1-bc).
\end{gather*}
The element $u=da-p(qbc+1)/(1-q)$ is normal regular in $M(p,q)$. It is not hard to check that the quotient algebra $N(p,q)=M(p,q)/(u)$ is realized as a GWA over $\kk[b,c]$. By Theorem \ref{thm:main-theorem} $N(p,q)$ is homologically smooth if and only if $p\neq 0$.

\begin{remark}
	The homological smoothness of $N(1,q)$ is studied by Shengyun Jiao, and the related results appear in her Master Thesis \cite{Jiao:Master-Thesis}, under the direction of the author.
\end{remark}

\subsection{Quantum lens space and quantum Seifert manifold}
Let us consider two algebras which can be regarded as coinvariant of Hopf algebras. They are the quantum lens space $\mc{O}_q(L(l;1,l))$ where $l$ is a positive integer, and the quantum Seifert manifold $\mc{O}_q(\Sigma^3)$. For their background, we refer to \cite{Hong-Szymanski:lens} and \cite{Brzezinski-Zielinski:Seifert} respectively. We should remark that both algebras were defined as $\con^*$-algebras originally; but here we adapt to an arbitrary base field $\kk$. Both of them are GWA as discussed in this paper. So they are homologically smooth by theorem \ref{thm:main-theorem}. This fact was first obtained by Brzezi{\'n}ski in \cite{Brzezinski:smooth-teardrop}, using a completely different manner.


\end{document}